\documentclass[11pt,reqno]{article}

\usepackage[letterpaper,hmargin=1in,top=1in,bottom=1.1in,
            footskip=0.6in]{geometry}

\usepackage{titlesec}
\titleformat{\subsection}[runin]{\normalfont\bfseries}
 {\thesubsection.}{.5em}{}[.]
\titlespacing{\subsection}{0pt}{2ex plus .1ex minus .2ex}{.8em}
\titleformat{\subsubsection}[runin]{\normalfont\bfseries}
 {\thesubsubsection.}{.3em}{}[.]
\titlespacing{\subsubsection}{0pt}{1ex plus .1ex minus .2ex}{.5em}
\titleformat{\paragraph}[runin]{\normalfont\itshape}
 {\theparagraph.}{.3em}{}[.]
\titlespacing{\paragraph}{0pt}{1ex plus .1ex minus .2ex}{.5em}

\usepackage{amsmath}
\usepackage{amssymb}
\usepackage{amsfonts}
\usepackage{latexsym}
\usepackage{amsthm}
\usepackage{amsxtra}
\usepackage{amscd}
\usepackage{bbm}
\usepackage{bm}
\usepackage{tensor}
\usepackage{microtype}
\usepackage{enumitem}
\usepackage{graphicx,color}
\definecolor{darkred}{rgb}{0.9,0,0.3}
\definecolor{darkblue}{rgb}{0,0.3,0.9}

\definecolor{vdarkred}{rgb}{0.7,0,0.2}
\definecolor{vdarkblue}{rgb}{0,0.2,0.7}
\usepackage[colorlinks,linkcolor=vdarkblue,citecolor=vdarkred,
            urlcolor=vdarkblue]{hyperref}

\usepackage[nottoc,notlof,notlot]{tocbibind}
\usepackage{cite}

\numberwithin{equation}{section}
\numberwithin{figure}{section}

\theoremstyle{plain}
\newtheorem{theorem}{Theorem}[section]
\newtheorem*{theorem*}{Theorem}
\newtheorem{lemma}[theorem]{Lemma}
\newtheorem*{lemma*}{Lemma}

\newtheorem*{corollary*}{Corollary}
\newtheorem{proposition}[theorem]{Proposition}
\newtheorem*{proposition*}{Proposition}
\newtheorem{conjecture}[theorem]{Conjecture}
\newtheorem*{conjecture*}{Conjecture}

\theoremstyle{definition}

\newtheorem*{definition*}{Definition}

\newtheorem*{example*}{Example}

\newtheorem*{remark*}{Remark}

\newtheorem*{assumption*}{Assumption}

\newcommand{\bb}{\mathbb}

\renewcommand{\P}{\mathbb{P}}
\newcommand{\E}{\mathbb{E}}
\newcommand{\R}{\mathbb{R}}
\newcommand{\C}{\mathbb{C}}

\newcommand{\Z}{\mathbb{Z}}
\newcommand{\Q}{\mathbb{Q}}
\newcommand{\e}{\mathrm{e}}

\newcommand{\ii}{\mathrm{i}}
\newcommand{\dd}{\mathrm{d}}

\newcommand*{\deq}{\mathrel{\vcenter{\baselineskip0.65ex \lineskiplimit0pt
 \hbox{.}\hbox{.}}}=}

\newcommand{\eqdist}{\overset{d}{=}}

\renewcommand{\leq}{\leqslant}

\renewcommand{\geq}{\geqslant}

\renewcommand{\epsilon}{\varepsilon}

\DeclareMathOperator{\diag}{diag}
\DeclareMathOperator{\tr}{Tr}

\DeclareMathOperator{\re}{Re}
\DeclareMathOperator{\im}{Im}

\newcommand{\cM}{{\mathcal M}}

\newcommand{\eqlabel}[1]{\refstepcounter{equation}\tag{\theequation}\label{#1}}
\newcommand{\F}{\mathbb F}
\newcommand{\Pp}{\P}
\newcommand{\one}{\mathbf 1}
\newcommand{\cG}{\mathcal G}
\newcommand{\eps}{\varepsilon}
\DeclareMathOperator{\Aff}{Aff}
\DeclareMathOperator{\rank}{rank}

\title{\bf \Large The oriented Kesten--McKay law
for random regular digraphs}

\author{Yukun He\footnote{Shanghai Center for Mathematical Sciences, Fudan University. Email:	\href{mailto:heyukun@fudan.edu.cn}{heyukun@fudan.edu.cn}}  \and Jiaoyang Huang\footnote{Department of Statistics and Data Science, University of Pensylvannia. Email:	\href{mailto:huangjy@wharton.upenn.edu}{huangjy@wharton.upenn.edu}} \vspace{1em}}

\begin{document}

\maketitle

\begin{abstract}
We consider the adjacency matrix of a random directed $d$-regular graph
on $N$ vertices.  For fixed $d\geq 2$, we prove that the empirical eigenvalue distribution converges weakly in probability to the
oriented Kesten--McKay law as $N\to \infty$. The key technical input is the small-ball probability estimate for the smallest singular value. The proof combines a fixed-rank transposition argument with
finite-field anticoncentration for shifted inverse compressions. We also prove a polynomial hard-edge estimate, which allows us to deduce the global law from the vanishing small-ball probability.
\end{abstract}


\section{Introduction and main results}\label{sec:introduction}

In this paper, we study random directed $d$-regular graphs on $N$ vertices under the uniform probability measure, without loops or multiple edges. Let $A\in \{0,1\}^{N\times N}$ be the adjacency matrix of the graph, and denote its eigenvalues by $\lambda_1,\ldots,\lambda_N$. 

For undirected random $d$-regular graphs on $N$ vertices, it is a classical result that when $d$ is fixed, the empirical eigenvalue density converges to the Kesten--McKay law
\[
\frac{d}{d^2-x^2}\frac{\sqrt{[4(d-1)-x^2]_+}}{2\pi}
\]
as $N\to\infty$ \cite{Kesten,McKay}. This follows from an elementary calculation via the moment method. For the directed model, the same question becomes much more challenging: the eigenvalues are complex, and complex moments alone do not determine their distribution. The following is a well-known problem.

\begin{conjecture}[Bordenave and Chafa\"{\i} \cite{BordenaveChafai}, Cook \cite{CookCircular}] \label{conj1.1}
	For fixed $d\geq 2$, the empirical eigenvalue measure $N^{-1}\sum_{i=1}^N\delta_{\lambda_i}$ converges weakly in probability to the measure with density
	\[
	h_d(z)\deq\frac{d^2(d-1)}{\pi(d^2-|z|^2)^2}
	\mathbf 1_{\{|z|<\sqrt d\}}
	\]
as $N\to \infty$. This measure is called the $\textnormal{oriented Kesten--McKay}$ law.
\end{conjecture}

We may now state our main result.

\begin{theorem}\label{thm:okm}
	Conjecture \ref{conj1.1} is true.
\end{theorem}

The most commonly used route to studying non-Hermitian random matrices is Girko's Hermitization \cite{Girko}. It relates the eigenvalues of a non-Hermitian matrix $B$ to the singular values of $B-w$, $w\in \bb C$, whose squares are the eigenvalues of the \textit{Hermitian} matrix $(B-w)(B^*-\bar{w})$. However, Girko's method requires a lower bound on the smallest singular value $s_{\mathrm{min}}(B-w)$. This is the central problem for non-Hermitian random matrices, and it is particularly difficult for graph models: the entries of the adjacency matrix are discrete, making it difficult to exclude bad events. In addition, the adjacency matrix of a regular graph is highly structured, and the methods used in the i.i.d.\,(Erd\H{o}s-R\'enyi) case often fail.

For $C<d<c_0N/(\log N\log\log N)$,
Litvak--Lytova--Tikhomirov--Tomczak-Jaegermann--Youssef \cite{LLTTYShifted} proved, in the model with loops allowed, that
\begin{equation}  \label{1}
\mathbb P(s_{\mathrm{min}}(A-w)<N^{-6})\leq d^{-1/4}
\end{equation}
for $|w|\leq d/6$, where $C$ is sufficiently large and $c_0>0$ is an absolute constant. Cook \cite{CookCircular} obtained related estimates in a polylogarithmic-degree regime. As the bound \eqref{1} decays as an inverse power of $d$, it was used in \cite{LLTTYCircular} to show that when $d \to \infty$, $N^{-1}\sum_{i=1}^N\delta_{\lambda_i/\sqrt{d}}$ converges weakly to the uniform measure on the unit disk. In other words, the spectrum in the growing degree regime satisfies the circular law \cite{Bai97,BR19,GT10,RT19,TaoVu,Wood}. For fixed $d\geq3$, Adhikari--Dembo \cite{AdhikariDembo} studied the eigenvalue distribution of the matrix $(A-w)(A^T-\bar{w})$, thereby reducing Conjecture \ref{conj1.1} to proving the statement
\begin{equation} \label{1.2}
	\bb P\big(s_{\mathrm{min}}(A-w)\leq e^{-N^{\kappa}}\big)=o(1)
\end{equation} 
for all fixed $\kappa>0$ and almost every $w\in \bb C$. However, the available estimates did not yield the vanishing small-ball probability needed for fixed $d$. 

The following is our second main result.

\begin{theorem}\label{thm:small}
	Fix an integer $d\geq2$ and $w\in\Q(i)\setminus\{0,d\}$.
	Then for every fixed $c>0$,
	\[
	\Pp\bigl(s_{\min}(A-wI)\leq e^{-cN}\bigr)=o_{d,w,c}(1)\,.
	\]
\end{theorem}

To prove Theorem \ref{thm:small} for $d\geq3$, our starting point is the finite-group mixing estimates in \cite{Meszaros}, which allow us to control shifted inverse compressions over finite fields. By reducing modulo suitable primes, we first show that the probability that $A-wI$ is singular goes to 0. To obtain the small-ball estimate, our main idea is to lift the smallest singular values by a bounded number of column transpositions. This is possible by combining an arithmetic determinant bound with graph expansion. Reversing the last transposition then gives an exponentially accurate relation involving the shifted inverse. We further use graph supersaturation and determinant identities to show that too many such relations would force an inverse compression to be singular. The finite-field estimates show that this event has vanishing probability. Comparing the numbers of forward and reverse paths then gives the desired bound. For $d=2$, we obtain the required finite-field estimates using an exact pairing formula and a lattice local limit theorem. 

We need to address three issues to deduce Theorem \ref{thm:okm} from Theorem \ref{thm:small}: (i) The condition \eqref{1.2} requires the bound for almost every $w$, not only Gaussian rationals; (ii) The exponent is any (small) positive power of $N$, not only linear; (iii) The work \cite{AdhikariDembo} does not cover the case $d=2$. Thus we need additional inputs to conclude Theorem \ref{thm:okm}.

To overcome (i), we approximate the integral w.r.t.\,the shift in Girko's Hermitization by a Riemann sum. It is important to perform this before transferring the eigenvalues to singular values, as the eigenvalues of $A-w$ have a clean, linear dependence on $w$. Since we are interested in the global law, it suffices to split the integration domain into a fixed number of regions. Thus we only need the small-ball estimate on any dense subset of the complex plane.

For issue (ii), we prove that for each fixed admissible $w$, the number of singular values of $A-w$ below $N^{-K}$ is bounded in probability for some fixed $K=K(d,w)$. Since $c>0$ in Theorem \ref{thm:small} can be arbitrarily small, the logarithm of the reciprocal of the smallest singular value is $o_{\mathbb P}(N)$. Thus the total contribution of these singular values to the normalized logarithmic determinant vanishes in probability. The remaining small singular values are controlled by the Hermitized local law in \cite{AdhikariDembo} for $d\geq3$.

For issue (iii), we prove the required averaged Green-function estimate for $d=2$. The proof uses a permutation representation of the degree-two model and the quantitative comparison estimates in \cite{CGVTH}. Together with the estimates in (ii), this gives logarithmic uniform integrability in probability at the fixed grid points and allows us to conclude the global law.

\subsection*{Other related works}
Haagerup and Larsen calculated the Brown measure of the sum of $d$ free
Haar unitaries and obtained the density $h_d$
\cite[Example~5.5]{HaagerupLarsen}.  Basak and Dembo proved convergence to
this measure for sums of independent Haar unitary or orthogonal matrices
\cite{BasakDembo}.  For sums of independent permutation matrices,
Basak, Cook, and Zeitouni proved the circular law in a growing-degree
regime \cite{BasakCookZeitouni}.

For independent uniform permutation matrices, Bordenave and Collins proved strong
asymptotic freeness on the orthogonal complement of the constant vector
\cite{BordenaveCollins}.  Coste proved the directed analogue of Alon's
spectral-gap conjecture \cite{Coste}.  Coste, Lambert, and Zhu studied the
fluctuations of the characteristic polynomial for sums of independent
permutation matrices.  They also obtained the spectral-gap result for
uniform simple regular digraphs by contiguity \cite{CosteLambertZhu}.

The unshifted invertibility problem was studied by Cook, Huang, and
Nguyen--Pan, among others
\cite{CookSingularity,HuangInvertibility,NguyenPan}.  The work
\cite{LLTTYStructure} studies the structure of approximate eigenvectors
of random regular digraphs.  In the sparse independent case, Sah,
Sahasrabudhe, and Sawhney proved the limiting law at fixed mean degree
using a singular-value-window method \cite{SahSahasrabudheSawhney}.

\subsection{Outline of proof}
\subsubsection*{Hermitization and the hard edge}
For an $N\times N$ matrix $B$, write $\tau_1(B)\leq\cdots\leq\tau_N(B)$
for its singular values, so that $s_{\min}(B)=\tau_1(B)$.  For $\lambda>0$,
put $n_\lambda(B)=\#\{j:\tau_j(B)<\lambda\}$.  For a fixed shift
$w\in\C$, put $C_w=A-wI$ and define
the symmetrized singular-value measure
\[
 \nu_N^w=\frac1{2N}\sum_{j=1}^N
 \bigl(\delta_{\tau_j(C_w)}+\delta_{-\tau_j(C_w)}\bigr)\,.
\]
Girko's identity relates the eigenvalue measure to the logarithmic potential
\[
 U_A(w):=\frac1N\log|\det(A-wI)|
        =\int_{\R}\log|x|\,\dd\nu_N^w(x)\,.
\]
Weak convergence of $\nu_N^w$ controls the integral away from zero, so it
remains to prove logarithmic uniform integrability at zero.  Write
$\log_+x=\max(\log x,0)$ and $F_{N,w}(y)=N^{-1}n_y(C_w)$ for $y>0$.
For $T>0$, put
\[
 R_{N,T}(w)=\frac1N\sum_{j=1}^N
 \left(\log_+\frac1{\tau_j(C_w)}-T\right)_+\,.
\]
Take $w\in\Q(i)\setminus\Z[i]$, so that $C_w$ is nonsingular.  For fixed
$K>0$ and $T<K\log N$, we have
\[
 \begin{split}
 R_{N,T}(w)
 =\int_T^{K\log N}F_{N,w}(e^{-u})\,\dd u+\frac1N\sum_{\tau_j(C_w)<N^{-K}}
       \log\frac{N^{-K}}{\tau_j(C_w)}\,.
 \end{split}
 \eqlabel{intro-three-window}
\]
For $d\geq3$, the local law of Adhikari and Dembo gives constants $C>0$ and
$\beta=\beta(d)>0$ such that, with probability tending to one,
\[
 F_{N,w}(y)\leq C(y\vee N^{-\beta})
\]
uniformly on compact sets of shifts.  Thus the first term in
\eqref{intro-three-window} is at most $Ce^{-T}+o(1)$.  For
$d=2$, we prove the corresponding integrated estimate below.

In Section~\ref{sec:count}, we prove that for some fixed $K=K(d,w)$,
\[
 n_{N^{-K}}(C_w)=O_{\Pp}(1)\,.
 \eqlabel{intro-poly-count}
\]
Together with Theorem~\ref{thm:small}, this controls the second term in
\eqref{intro-three-window}, since
\[
 0\leq\frac1N\sum_{\tau_j(C_w)<N^{-K}}
       \log\frac{N^{-K}}{\tau_j(C_w)}
 \leq\frac{n_{N^{-K}}(C_w)}N
       \log_+\frac1{s_{\min}(C_w)}=o_{\Pp}(1)\,.
\]
Finally, Lemma~\ref{lem:grid} approximates Girko's integral by a finite
Riemann sum with nodes in $\Q(i)\setminus\Z[i]$.  The estimates above give
convergence of $U_A$ simultaneously at these nodes.  We first choose the
truncation level $T$ large and the mesh small, both independently of $N$,
and then let $N\to\infty$.  This proves Theorem~\ref{thm:okm} using only
the arithmetic shifts covered by Theorem~\ref{thm:small}.

\subsubsection*{The least singular value}
The transposition arguments are carried out in the uniform model with
loops allowed.  Fix $c>0$ and $w\in\Q(i)\setminus\{0,d\}$.  Choose a positive
integer $q$ with $qw\in\Z[i]$ and put $C=A-wI$.  On the event
$0<s_{\min}(C)\leq e^{-cN}$, the arithmetic bound
$|\det C|\geq q^{-N}$ and the operator norm bound $\|C\|\leq d+|w|$
isolate a bounded number of small singular values.  More precisely, for
$b>1$ sufficiently close to one, we find an integer $r$ bounded
independently of $N$ and a scale $h$ comparable to $N$ such that
\[
 \tau_r(C)\leq e^{-h}\,,
 \qquad
 \tau_{r+1}(C)\geq e^{-h/b}\,.
\]
We use a weighted decomposition of the logarithms of the singular values
to integrate over the possible gaps.

Let $e_i$ be the $i$th standard basis vector.  For an unordered pair
$e=\{i,j\}$, put
$a_e=e_i-e_j$, $P_e=I-a_ea_e^*$, and $T_eA=AP_e$.  This exchanges columns
$i$ and $j$, with
\[
 T_eA-wI=CP_e-wa_ea_e^*\,.
\]
Small-set expansion shows that a positive proportion of the vectors $a_e$
have a nontrivial projection onto both bottom singular subspaces.  A path
of $r$ such transpositions lifts the small singular values to give a shifted
endpoint $B$ with $s_{\min}(B)\geq N^{-O(1)}e^{-h/b}$.  Reversing the last
transposition gives
\[
 \left|1+w a_e^*B^{-1}a_e\right|
 \leq N^{O(1)}e^{-(1-1/b)h}\,.
 \eqlabel{intro-resonance}
\]
For $d\geq3$, M\'esz\'aros's finite-group mixing theorem \cite{Meszaros}
controls fixed shifted inverse compressions modulo suitable primes.
Gaussian-integer separation and graph supersaturation
\cite{ErdosSimonovits} then show that the expected number of pairs satisfying
\eqref{intro-resonance} is $o(N^2)$ under the uniform law.  Comparing the
numbers of forward and reverse paths transfers this bound to the initial
matrix.  The finite-field estimate also excludes exact singularity.  Integration
over $h$ completes the proof, and conditioning gives the zero-diagonal case.

\subsubsection*{The polynomial hard edge}
We prove \eqref{intro-poly-count} in two stages.  First, we rule out
approximate kernel vectors of polynomially large norm with a level set
containing a positive proportion of their coordinates.  If such vectors
existed, rescaling and graph expansion would give a limiting row law
$(X,Y_1,\ldots,Y_d)$ satisfying
\[
 Y_1+\cdots+Y_d=wX\,,
\]
where all marginals have the same nonconstant law with an atom.  We find a
finite Borel partition $\mathcal P$ with the following property.  Writing
$\mathcal P(X)$ for the cell containing $X$ and $H$ for Shannon entropy,
\[
 H(\mathcal P(X),\mathcal P(Y_1),\ldots,\mathcal P(Y_d))
       -dH(\mathcal P(X))<0\,.
\]
Colouring the coordinates by this partition contradicts an exponential
counting bound in the configuration model.  This gives an augmented
rectangular invertibility estimate for $A-wI$ and its adjoint.

The second stage turns this estimate into a count of small singular values.
Set $C=A-wI$, $\lambda=N^{-K}$, and
\[
 \Psi_\lambda(A)=\log\det(C^*C+\lambda^2I)\,.
\]
Fix an integer $s\geq1$.  On the event where the rectangular invertibility
estimate holds, if $n_\lambda(C)\geq s$, unit vectors in the bottom left and
right singular subspaces have spread coordinates.  Hence both projection
norms have a polynomial lower bound for more than half of the transposition
labels.  A regularized determinant
identity shows that $\Psi_\lambda$ increases along these labels.  Orienting
the transposition edges in the direction of increase and counting the
boundary of $\{A:n_\lambda(A-wI)\geq s\}$ gives, for some $\vartheta>1/2$,
\[
 \limsup_{N\to\infty}
 \Pp\{n_\lambda(A-wI)\geq s\}\leq(2\vartheta)^{-s}\,.
\]
This proves the tightness in \eqref{intro-poly-count}, which transfers to
the zero-diagonal model by conditioning.

\subsubsection*{Degree two}
For $d=2$, we prove the finite-field estimates using an exact pairing formula
and a lattice local limit theorem in Subsection~\ref{subsec:d2}.  To obtain
the local estimate, let $P$ be a uniform permutation matrix and $Q$ an
independent permutation matrix whose law assigns weight $2^{-k}$ to each
permutation with $k$ cycles and no fixed points.  In the uniform degree-two
model with loops allowed,
\[
 A\eqdist P(I+Q)\,.
\]
We couple $Q$ to a uniform permutation so that the expected rank of their
difference is $O(\log N)$, and apply the comparison theorem of Chen--Garza-Vargas--Tropp--van Handel
\cite{CGVTH}.  With $\tr$ denoting the unnormalized trace, put, for $\eta>0$,
\[
 H_w=\begin{pmatrix}0&C_w\\ C_w^*&0\end{pmatrix}\,,
 \qquad
 m_{N,w}(\ii\eta)=\frac1{2N}\tr(H_w-\ii\eta)^{-1}\,.
\]
Computing the Green function of the limiting free operator gives, uniformly on
compact sets of shifts and for $0<\eta\leq1$,
\[
 \E\,\im m_{N,w}(\ii\eta)
 \leq C+\frac{C}{N\eta^6}+\frac{C\log N}{N\eta}\,.
\]
Taking $\eta=N^{-1/7}$ gives the mesoscopic estimate needed in
\eqref{intro-three-window}.  Conditioning transfers the result to the
zero-diagonal model.

\subsection*{Organization} Section~\ref{sec:inputs} proves the finite-field compression estimates,
model transfer, and expansion.  Section~\ref{sec:compression} proves the
shifted inverse-compression estimate and excludes reverse resonance.
Section~\ref{sec:fixed-degree} combines these inputs along transposition
paths to prove Theorem~\ref{thm:small}.  Section~\ref{sec:count} proves the
polynomial hard-edge count.  Section~\ref{sec:girko} establishes logarithmic
uniform integrability and completes the proof of Theorem~\ref{thm:okm}.

\subsection*{Acknowledgment} YH is supported by NSFC No.\,12322121 and the National Key R\&D Program of China No.\,2023YFA1010400. The research of JH is supported by NSF grants DMS-2331096 and DMS-2337795, and the Sloan Research Award.

\subsection*{LLM disclosure} GPT-5.6 assisted the proof development at the level of a coauthor. The authors are responsible for the content.

\section{Finite-field compression, model transfer, and expansion}
\label{sec:inputs}

\subsection*{Notation}
We write $\mathbb N_+=\{1,2,\ldots\}$, $[N]=\{1,\ldots,N\}$, and
$\one=(1,\ldots,1)^{\mathsf T}$.  The norm $\|\cdot\|$ is the Euclidean
norm for vectors and the operator norm for matrices.  We use the two
matrix models
\[
 \cM_{N,d}=\{B\in\{0,1\}^{N\times N}:
 B\one=d\one\,,\ B^{\mathsf T}\one=d\one\}\,,
 \qquad
 \cM^0_{N,d}=\{B\in\cM_{N,d}:\diag B=0\}\,.
\]
The first allows self-loops, while the second is the simple,
zero-diagonal model used in the main results.  Throughout the proof,
$d$ is fixed and $N\to\infty$.  The shift $w=d$ is excluded since
$A\one=d\one$.

For an $N\times N$ matrix $B$, let $\lambda_j(B)$ denote its eigenvalues
counted with algebraic multiplicity.  We write
\[
 L_B=\frac1N\sum_{j=1}^N\delta_{\lambda_j(B)}\,,
\]
and denote by $\mu_d$ the probability measure with density $h_d$.

For $x\in\R$ and $k\in\mathbb N_+$, write
$(x)_k=x(x-1)\cdots(x-k+1)$, with $(x)_0=1$.

\subsection{A finite-field compression}

For $d\geq2$, give each row vertex and each column vertex $d$ labelled
half-edges, match the two sets of $dN$ half-edges uniformly, and let
$\widehat A$ be the resulting biadjacency matrix.  Its entries may
exceed one.  For $d\geq3$, let $P_1,\ldots,P_d$ be independent uniform
permutation matrices.  Put
\[
 \mathcal A_{N,d}=
 \begin{cases}
  \widehat A\,,&d=2\,,\\
 P_1+\cdots+P_d\,,&d\geq3\,.
 \end{cases}
\]
The finite-field argument uses $\mathcal A_{N,d}$, while augmented
expansion and the polynomial hard-edge argument use $\widehat A$.

For $z\in V^N$, let $\Aff\{z_1,\ldots,z_N\}$ denote the smallest
coset of a subgroup of $V$ containing all coordinates of $z$.  We use
the following specialization of Theorem~1.5 of M\'esz\'aros
\cite{Meszaros}.  If $d\geq3$ and $V$ is a fixed finite abelian group,
then
\[
 \sum_{\substack{z=(z_1,\ldots,z_N)\in V^N:\
                  \Aff\{z_1,\ldots,z_N\}=V}}
 \max_{y\in V^N}
 \left|
 \Pp((P_1+\cdots+P_d)z=y)-
 \frac{\mathbf 1_{\{\sum_jy_j=d\sum_jz_j\}}}{|V|^{N-1}}
 \right|=o(1)\,.
 \eqlabel{meszaros}
\]
We use this summed error below.  For $d=2$, the shifted form is proved in
Subsection~\ref{subsec:d2}.

For an integer $\ell\geq1$, let $X,Y\subset[N]$ be disjoint sets of
size $\ell+1$, choose anchors $x_0\in X$ and $y_0\in Y$, and define
the $N\times\ell$ difference matrices
\[
 L_X=(e_x-e_{x_0})_{x\in X\setminus\{x_0\}}\,,
 \qquad
 L_Y=(e_y-e_{y_0})_{y\in Y\setminus\{y_0\}}\,.
 \eqlabel{difference-matrices}
\]

\begin{proposition}
\label{prop:ff-comp}
Fix $d\geq2$, an odd prime $p$, an integer $\ell\geq1$, and
\[
 \omega\in\F_p\setminus\{0,d\}\,.
\]
Reduce $L_X,L_Y$ modulo $p$ and put
$E=(L_X,L_Y)\in\F_p^{N\times2\ell}$.  Uniformly over the disjoint
anchored sets,
\[
 \Pp\left(
 \begin{array}{c}
 \mathcal A_{N,d}-\omega I\text{ is invertible over }\F_p\,,\\
 \det L_X^{\mathsf T}\bigl((\mathcal A_{N,d}-\omega I)^{-1}
 +(\mathcal A_{N,d}-\omega I)^{-\mathsf T}\bigr)L_Y=0
 \end{array}
 \right)
 \leq\frac1{p-1}+o(1)\,.
 \eqlabel{ff-proj}
\]
\[
 \Pp(\mathcal A_{N,d}-\omega I\text{ is singular over }\F_p)
 \leq\frac1{p-1}+o(1)\,.
 \eqlabel{ff-sing}
\]
The errors may depend on $d$, $p$, and $\ell$.
\end{proposition}

\begin{proof}
We first prove the proposition for $d\geq3$.  The case $d=2$ is
proved in Subsection~\ref{subsec:d2}.  Put
\[
 r=2\ell\,,\qquad V=\F_p^r\,,
\]
and retain $E=(L_X,L_Y)$ for its reduction modulo $p$.  The rows of
$E$ include zero and every standard basis vector, while every column
has coordinate sum zero.  Hence
\[
 \Aff\{E_1,\ldots,E_N\}=V\,,
 \qquad \one^{\mathsf T}E=0\,.
\]
For $T\in\F_p^{r\times r}$,
let $Z_T$ count
the matrices $Z\in\F_p^{N\times r}$ satisfying
\[
 ((P_1+\cdots+P_d)-\omega I)Z=E\,,
 \qquad E^{\mathsf T}Z=T\,.
 \eqlabel{finite-field-ZT}
\]
Every solution satisfies
\[
 (d-\omega)\one^{\mathsf T}Z=\one^{\mathsf T}E=0\,,
 \qquad \one^{\mathsf T}Z=0\,.
 \eqlabel{ff-cons}
\]
Its rows also affinely span $V$.  Indeed, if the affine span of its
rows is $a+H$, then the rows of the left side of
\eqref{finite-field-ZT} lie in $(d-\omega)a+H$.  The rows of $E$
affinely span $V$, so $H=V$.

The affine-span assumption also implies
$\rank[\one,E]=r+1$.  Hence the fiber
\[
 \{Z:\one^{\mathsf T}Z=0\,,\ E^{\mathsf T}Z=T\}
\]
has $p^{r(N-r-1)}$ elements, of which at most
$O_{p,r}(p^{(r-1)N})$ have rows in a proper affine subspace.  For every
other $Z$, the target $E+\omega Z$ belongs to the support in
\eqref{meszaros}.  Summing \eqref{meszaros} over the fiber gives
\[
 \E Z_T
 =p^{r(N-r-1)}p^{-r(N-1)}
  +O_{p,r}(p^{-N+r})+o(1)
 =p^{-r^2}+o(1)\,.
\]
On the invertible event the solution is unique.  Thus, uniformly in
$T\in\F_p^{r\times r}$,
\[
 \Pp\bigl(\mathcal A_{N,d}-\omega I\text{ is invertible}\,,
 E^{\mathsf T}(\mathcal A_{N,d}-\omega I)^{-1}E=T\bigr)
 \leq p^{-r^2}+o(1)\,.
 \eqlabel{finite-field-point}
\]
Write $T$ in four $\ell\times\ell$ blocks.  The projected matrix in
\eqref{ff-proj} equals $T_{12}+T_{21}^{\mathsf T}$.
For uniform $T$ this is a uniform $\ell\times\ell$ matrix, and
\[
 \Pp(\det(T_{12}+T_{21}^{\mathsf T})=0)
 =1-\prod_{j=1}^{\ell}(1-p^{-j})\leq\frac1{p-1}\,.
\]
Summing \eqref{finite-field-point} over these $T$ proves
\eqref{ff-proj}.

For the singularity bound, a nonzero kernel vector is nonconstant
because $\omega\neq d$, and \eqref{ff-cons} gives
$\sum_jz_j=0$.  Conversely, every nonconstant vector with this sum has
full affine span in $\F_p$.  Hence \eqref{meszaros} gives
\[
 \begin{split}
 \E\#\{0\neq z:(\mathcal A_{N,d}-\omega I)z=0\}
 &=\sum_{\substack{z:\ \Aff\{z_1,\ldots,z_N\}=\F_p\\
                         \sum_jz_j=0}}
 \Pp((\mathcal A_{N,d}-\omega I)z=0)\\
 &=(p^{N-1}+O_p(1))p^{-(N-1)}+o(1)=1+o(1)\,,
 \end{split}
 \eqlabel{finite-field-kernel}
\]
Every nontrivial kernel contains at least $p-1$ nonzero vectors, so
\eqref{finite-field-kernel} proves \eqref{ff-sing}.
\end{proof}

\subsection{Degree two over finite fields}\label{subsec:d2}

\begin{proof}[Proof of Proposition~\ref{prop:ff-comp} for $d=2$]
Fix $p,\ell,\omega,X$, and $Y$ as in
Proposition~\ref{prop:ff-comp}, let
$T\in\F_p^{2\ell\times2\ell}$, and put
\[
 r=2\ell\,,\qquad V=\F_p^r\,,\qquad E=(L_X,L_Y)\,.
\]
The matrix $E$ has at most $2\ell+2$ nonzero rows,
$\one^{\mathsf T}E=0$, and its rows affinely span $V$.  All estimates
below are uniform over the anchored sets and $T$.
Constants denoted by $C$ below may depend on the fixed
$p$, $\ell$, and $\omega$.

Write $Q=|V|$.  We estimate the entropy away from
the uniform law on $V$ and prove a local limit near it.  We first derive the
pairing formula.  For fixed
$p$ and $\ell$, simultaneous permutation invariance moves the support
of $E$ into a fixed set of $2\ell+2$ rows, after which only finitely
many $E$ remain.  It is therefore enough to prove the estimates for
each fixed $E$.

For
$z\in V^N$, let $n_a=\#\{i:z_i=a\}$.  Let
$(U_i,V_i)_{i\in[N]}$ be independent pairs of independent uniform
$V$-valued variables, and put
\[
 \mathcal C_a=\sum_{i=1}^N
   (\mathbf1_{\{U_i=a\}}+\mathbf1_{\{V_i=a\}})\,,
 \qquad S_i=U_i+V_i\,.
\]
Write $\mathcal C=(\mathcal C_a)_{a\in V}$,
$n=(n_a)_{a\in V}$, and $S=(S_i)_{i\in[N]}$.
Conditioning on $\mathcal C=2n$ makes the ordered list of the $2N$
variables a uniform ordering of the multiset containing two copies of
each $z_i$.  Bayes' formula therefore gives, for every $y\in V^N$,
\[
 \Pp(\widehat Az=y)
 =Q^{-N}\frac{\Pp(\mathcal C=2n\mid S=y)}
                  {\Pp(\mathcal C=2n)}\,.
 \eqlabel{d2-pairing-formula}
\]

We next identify the zero set of the entropy rate.  Put
$N_\eta=\#\{i:E_i=\eta\}$.  For $\eta,a,b\in V$, let
$m_{\eta,a,b}$ count the rows with $E_i=\eta$ and ordered stub
colours $(a,b)$; on such a row,
$z_i=\omega^{-1}(a+b-\eta)$.  Write
$m=(m_{\eta,a,b})$.  Thus
\[
 \sum_{a,b}m_{\eta,a,b}=N_\eta\,,
 \qquad
 n_c=\sum_{\eta,a+b-\eta=\omega c}m_{\eta,a,b}\,,
 \qquad
 \sum_{\eta,b}m_{\eta,c,b}+\sum_{\eta,a}m_{\eta,a,c}=2n_c\,,
\]
and therefore
\[
 \E\#\{z\in V^N:(\widehat A-\omega I)z=E\}
 =\sum_m
 \frac{\prod_\eta N_\eta!}
      {\prod_{\eta,a,b}m_{\eta,a,b}!}
 \frac{\prod_c(2n_c)!}{(2N)!}\,,
 \eqlabel{d2-inhom}
\]
where the sum is over the nonnegative integer tables satisfying the
three displayed constraints.  Here $H$, $I$, and $D$ denote Shannon
entropy, mutual information, and relative entropy, respectively.
Let $(\eta,U,V)$ have empirical law
$m_{\eta,a,b}/N$.  Since $\Pp(\eta\ne0)=O(N^{-1})$,
$H(\eta)=O((\log N)/N)$ and
\[
 |H(U,V\mid\eta)-H(U,V)|\leq H(\eta)
 =O((\log N)/N)\,.
\]
Thus replacing the conditional table entropy in
\eqref{d2-inhom} by the entropy of the pair below
changes the logarithm by $O(\log N)$, uniformly over all tables.

Let $\gamma$ be the empirical law of the pair $(U,V)$,
let $\alpha,\beta$ be its marginals, and let $\pi$ be the empirical
law of $z$.  Stub balance gives $(\alpha+\beta)/2=\pi$.  Stirling's
formula, applied to the fixed finite set of table cells, shows that the
logarithm of each summand in \eqref{d2-inhom} is $N$ times
\[
 H(\gamma)-2H(\pi)=-\mathcal D(\gamma,\pi)\,,
\]
up to $O(\log N)$, where the identity
\[
 \mathcal D(\gamma,\pi)
 =I_\gamma(U;V)+D(\alpha\Vert\pi)+D(\beta\Vert\pi)
 \geq0
 \eqlabel{d2-entropy-identity}
\]
follows from $(\alpha+\beta)/2=\pi$.  In a limiting table with
$\mathcal D=0$, the variables $U,V$ are independent with common law
$\pi$.  Write $*$ for convolution on $V$ and $\omega_*\pi$ for the
pushforward of $\pi$ under $a\mapsto\omega a$.  The row equation gives
\[
 \pi*\pi=\omega_*\pi\,.
 \eqlabel{d2-idem}
\]

The solutions of \eqref{d2-idem} are precisely the
uniform laws on linear subspaces of $V$.  To see this, Fourier
transformation gives
\[
 \widehat\pi(\chi)^2=\widehat\pi(\omega\chi)\,.
\]
Multiplication by $\omega$ permutes the finite dual group.  Iteration
around each orbit shows that the absolute value of each Fourier coefficient is zero or one.
The characters whose coefficients have absolute value one form a subgroup, and equality in the
triangle inequality then shows that $\pi$ is uniform on a coset
$a+H$.  Comparing supports in \eqref{d2-idem} gives
$(2-\omega)a\in H$.  Since $2-\omega$ is invertible, $a\in H$.

We shall also need quantitative stability near the uniform law on $V$.
Let $\upsilon$ denote the uniform law on $V$ and write
$\pi=\upsilon+f$.  Pinsker's inequality in
\eqref{d2-entropy-identity} gives
\[
 \mathcal D\geq c_V\bigl(
 \|\alpha-\beta\|_1^2+
 \|\gamma-\alpha\otimes\beta\|_1^2\bigr)\,.
 \eqlabel{d2-Pinsker}
\]
Let $\gamma_+$ be the pushforward of $\gamma$ under $(a,b)\mapsto a+b$.
Since $E_i\neq0$ on only $O(1)$ rows,
\[
 \|\gamma_+-\omega_*\pi\|_1=O(N^{-1})\,.
\]
Writing $\alpha=\pi+h$ and $\beta=\pi-h$, we have
\[
 \alpha*\beta=\pi*\pi-h*h\,.
\]
Contraction under the sum map and \eqref{d2-Pinsker} therefore yield
\[
 \begin{split}
 \|\omega_*\pi-\pi*\pi\|_1
 \leq O(N^{-1})+\|\gamma-\alpha\otimes\beta\|_1+\|h*h\|_1\leq C_V(\sqrt{\mathcal D}+\mathcal D)+O(N^{-1})\,.
 \end{split}
\]
On the other hand,
$\pi*\pi=\upsilon+f*f$ and
$\omega_*\pi=\upsilon+\omega_*f$, so
\[
 \|\omega_*\pi-\pi*\pi\|_1
 =\|\omega_*f-f*f\|_1
 \geq\|f\|_1-\|f\|_1^2\,.
\]
Hence, in a fixed small neighbourhood of $\upsilon$,
\[
 \mathcal D\geq c_{V,\omega}
 \biggl(\|\pi-\upsilon\|_1-\frac{C_{V,\omega}}N\biggr)_+^2
 \eqlabel{d2-deficit}
\]
Away from small neighbourhoods of the finitely many uniform laws on
subspaces, compactness and the preceding classification give a fixed
positive entropy gap.  Indeed, after choosing disjoint sufficiently small
neighbourhoods of the finitely many laws uniform on subspaces
$H\leq V$, the feasible limiting tables in
their complement form a compact set on which $\mathcal D$ has a
strictly positive minimum.  Equation \eqref{d2-deficit}
applies near the uniform law on $V$; the following
quotient estimate applies near the uniform laws on proper subspaces.

Let $W$ be a nonzero
quotient of $V$, let $e\in W^N$ have support in a fixed set of size
$s$, and fix $x\in W^N$.  Put
\[
 k=|\operatorname{supp}x|\,,
 \qquad l=|\operatorname{supp}(\omega x+e)|\,.
\]
The $2k$ nonzero quotient-coloured column stubs occupy a uniform
$2k$-subset of the $2N$ ordered row slots.  If
$(\widehat A-\omega I)x=e$, every row in
$\operatorname{supp}(\omega x+e)$ receives at least one such stub,
whereas every hit row outside this support receives two.  If $c$ of
the former rows receive two nonzero stubs, then
$j=k-(l+c)/2$ rows outside the target support receive two.  Therefore
\[
 \Pp((\widehat A-\omega I)x=e)
 \leq\frac{1}{\binom{2N}{2k}}
 \sum_{c=0}^l
 \mathbf1_{\{2k-l-c\in2\Z_{\geq0}\}}
 \binom lc2^{l-c}\binom{N-l}{k-(l+c)/2}\,.
 \eqlabel{d2-support-bound}
\]
For $1\leq k\leq N/4$, the standard binomial estimates give
\[
 \Pp((\widehat A-\omega I)x=e)
 \leq C^k(k/N)^{k+l/2}\,.
 \eqlabel{d2-supp-simple}
\]
Indeed, in a term of \eqref{d2-support-bound} the denominator leaves
the power
$2k-j=k+(l+c)/2$, and $(\e k/j)^j\leq C^k$.

Write $k_0=|\operatorname{supp}x\setminus\operatorname{supp}e|$
and $k_1=k-k_0$.  Since $\omega\neq0$, one has $l\geq k_0$.
After summing \eqref{d2-supp-simple} over the support and the
nonzero values of $x$, one obtains
\[
 \sum_{\substack{x:\ |\operatorname{supp}x\setminus
                         \operatorname{supp}e|=k_0\\
                         |\operatorname{supp}x\cap
                         \operatorname{supp}e|=k_1}}
 \Pp((\widehat A-\omega I)x=e)
 \leq C_{W,s}^{k+1}
 \left(\frac{k}{k_0\vee1}\right)^{k_0}
 \left(\frac{k}{N}\right)^{k_0/2+k_1}\,.
 \eqlabel{d2-support-sum}
\]
Indeed, there are at most $\binom N{k_0}2^s$ choices of the support
and at most $(|W|-1)^k$ choices of its nonzero values.  Inserting
$\binom N{k_0}\leq(\e N/k_0)^{k_0}$ into
\eqref{d2-supp-simple}, and using $l\geq k_0$, gives
\eqref{d2-support-sum} after changing $C_{W,s}$.
For $k_0\geq2s+1$, this is bounded by
$C_{W,s}(C_{W,s}\sqrt{k_0/N})^{k_0}$.  Choosing a sufficiently small
fixed $0<\delta<1/4$ makes its sum over
$2s+1\leq k_0\leq\delta N$ tend to zero.  The remaining finitely many
values of $k_0$ contribute $O(N^{-1/2})$; when $k_0=0$, the zero
vector is impossible if $e\neq0$, and every remaining term is
$O(N^{-1})$.  We have proved
\[
 \Pp\bigl(\exists x:\ |\operatorname{supp}x|\leq\delta N\,,
       (\widehat A-\omega I)x=e\bigr)=o(1)
 \eqlabel{d2-sparse-quotient}
\]
whenever $e\neq0$ has bounded support.
Since $V$ has only finitely many quotients, the same $\delta$ may and
will be chosen for all of them.

We next prove the local limit near the uniform law on $V$.  Fix
$
 \frac12<\theta<\frac23.
$
Uniformly for
\[
 \max_a|n_a-N/Q|\leq N^\theta\,,
 \qquad
 \max_a|\#\{i:y_i=a\}-N/Q|\leq2N^\theta\,,
 \qquad \sum_i y_i=2\sum_i z_i\,,
 \eqlabel{d2-local-window}
\]
we claim that
\[
 \frac{\Pp(\mathcal C=2n\mid S=y)}
      {\Pp(\mathcal C=2n)}=Q+o(1)\,.
 \eqlabel{d2-local-ratio}
\]
We prove the claim by a uniform saddle-point calculation.  Fix a base
point $a_0\in V$, put $t_{a_0}=0$, and work in the remaining
$Q-1$ real coordinates.  The one-pair moment generating functions are
\[
 \Phi_{\mathrm{unc}}(t)=\left(Q^{-1}\sum_{a\in V}e^{t_a}\right)^2\,,
 \qquad
 \Phi_c(t)=Q^{-1}\sum_{a\in V}e^{t_a+t_{c-a}}\,.
 \eqlabel{d2-moment-functions}
\]
Thus the unconditional and conditional log moment generating functions
of the count vector are
\[
 K_{\mathrm{unc}}(t)=N\log\Phi_{\mathrm{unc}}(t)\,,
 \qquad K_y(t)=\sum_{i=1}^N\log\Phi_{y_i}(t)\,.
\]
Their gradients at zero are both $2NQ^{-1}\one$.  Moreover,
\[
 \frac1Q\sum_{c\in V}\nabla^2\log\Phi_c(0)
 =\nabla^2\log\Phi_{\mathrm{unc}}(0)=:\Sigma_*\,.
\]
Indeed, if $c$ and $U$ are independent uniform variables on $V$, then
$(U,c-U)$ is also a pair of independent uniform variables.  The matrix $\Sigma_*$ is positive
definite in the chosen $Q-1$ coordinates.

Let $x$ be the vector formed by the corresponding coordinates of
$2n$.  Uniformly in \eqref{d2-local-window}, strict convexity and the
inverse function theorem give unique real saddles $t_{\mathrm{unc}},t_y$ with
\[
 \nabla K_{\mathrm{unc}}(t_{\mathrm{unc}})=x\,,
 \qquad \nabla K_y(t_y)=x\,,
 \qquad |t_{\mathrm{unc}}|+|t_y|=O(N^{\theta-1})\,.
 \eqlabel{d2-real-saddles}
\]
The values and first two derivatives at zero of
$Q^{-1}\sum_c\log\Phi_c$ and $\log\Phi_{\mathrm{unc}}$ agree.  Since every third
derivative of the fixed functions in \eqref{d2-moment-functions} is
bounded near zero, the second bound in \eqref{d2-local-window} gives,
uniformly for
$|t|\leq C N^{\theta-1}$,
\[
 |K_y(t)-K_{\mathrm{unc}}(t)|
 \leq C(N|t|^3+N^\theta|t|^2)
 =O(N^{3\theta-2})=o(1)\,.
 \eqlabel{d2-rate-comparison}
\]
Consequently the two Legendre transforms
\[
 I_j(x)=\langle t_j,x\rangle-K_j(t_j)\,,
 \qquad j\in\{\mathrm{unc},y\}\,,
\]
satisfy
\[
 |I_y(x)-I_{\mathrm{unc}}(x)|
 \leq\max_{t\in\{t_{\mathrm{unc}},t_y\}}
 |K_y(t)-K_{\mathrm{unc}}(t)|=o(1)\,.
\]
Moreover,
\[
 N^{-1}\|\nabla^2K_y(t_y)-\nabla^2K_{\mathrm{unc}}(t_{\mathrm{unc}})\|
 \leq C(N^{\theta-1}+|t_y|+|t_{\mathrm{unc}}|)=o(1)\,,
 \eqlabel{d2-Hess}
\]
so the determinants in the two Gaussian approximations have ratio $1+o(1)$.

It remains to compute the lattice prefactor uniformly.  Exponentially
tilt the two arrays by $t_{\mathrm{unc}},t_y$, respectively.  Their means are now
$x$, their summands remain uniformly bounded, and their normalized
covariances converge to the positive definite matrix $\Sigma_*$.
The unconditional difference lattice is
\[
 L_{\mathrm{unc}}=\{h\in\Z^V:\sum_ah_a=0\}\,,
\]
and for the conditional count vector we define
\[
 L=\Big\{h\in L_{\mathrm{unc}}:\sum_{a\in V}a h_a=0\text{ in }V\Big\}\,.
\]
Let $L_y$ be the lattice generated by
$(e_a+e_{c-a})-(e_b+e_{c-b})$, where
$c\in\{y_1,\ldots,y_N\}$ and $a,b\in V$; in particular, each
difference has $c$ fixed.  The second condition in
\eqref{d2-local-window} ensures that
every $c\in V$ occurs for large $N$, and hence $L_y\subset L$.
To prove equality, note that exponential tilting does not change the
unit-modulus Fourier saddles.  Represent a character of $L_{\mathrm{unc}}/L_y$ by
phases $(\xi_a)_{a\in V}$, determined up to a common phase.  It satisfies
\[
 \xi_a\xi_{c-a}=\xi_0\xi_c\,,
 \qquad a,c\in V\,.
\]
Taking $c=a+b$ shows that $a\mapsto\xi_a/\xi_0$ is a character of
$V$.  Hence, after the common
phase has been removed, the conditional integral has exactly $Q$
saddles and the unconditional integral has one.  Conversely, every
character produces such a saddle.  Since $[L_{\mathrm{unc}}:L]=Q$, it follows that
$L_y=L$ and that the conditional lattice prefactor is $Q$ times the
unconditional one.

We now estimate the Fourier inversion integral.  Choose
$0<\varepsilon<1/6$.  In a ball of radius
$N^{-1/2+\varepsilon}$ around any saddle, Taylor expansion of the
logarithm of the centered tilted characteristic function gives
\[
 -\frac N2\langle s,\Sigma_j s\rangle+O(N|s|^3)\,,
 \qquad
 \Sigma_j=N^{-1}\nabla^2K_j(t_j)\,,
\]
and $N|s|^3=o(1)$ throughout the ball.  After the change of variables
$v=N^{1/2}s$, its integral is therefore
\[
 (2\pi N)^{-(Q-1)/2}(\det\Sigma_j)^{-1/2}(1+o(1))\,.
 \eqlabel{d2-local-int}
\]
On the annulus from $N^{-1/2+\varepsilon}$ to a fixed small radius,
uniform positive definiteness gives the bound
$\exp(-cN|s|^2)$, whose integral is
$O(e^{-cN^{2\varepsilon}})$.  Outside the fixed saddle
neighbourhoods, every $c\in V$ occurs at least $N/(2Q)$ times for
large $N$; compactness and \eqref{d2-moment-functions} then give
$O(e^{-cN})$.  These bounds are unchanged by the real tilts in
\eqref{d2-real-saddles}.

At the saddle indexed by a character $\chi$, the Fourier phase is
\[
 \chi\left(\sum_i y_i-2\sum_i z_i\right)=1
\]
by the last condition in \eqref{d2-local-window}.  Thus the $Q$ conditional saddle contributions add
with the same phase.  Equations \eqref{d2-rate-comparison},
\eqref{d2-Hess}, and
\eqref{d2-local-int} prove
\eqref{d2-local-ratio}.
Combining it with \eqref{d2-pairing-formula} gives
\[
 \Pp(\widehat Az=y)=Q^{-(N-1)}(1+o(1))
 \eqlabel{d2-central-mixing}
\]
uniformly in \eqref{d2-local-window}.

We now prove the compression estimate.  Since $V$ and the set of values of $E_i$ are
fixed, the number of empirical tables is polynomial in $N$, and the
uniform Stirling error for each table is $O(\log N)$.
By \eqref{d2-deficit} and Stirling's formula, the total
contribution of tables near the uniform law but outside
\eqref{d2-local-window} is at most
\[
 N^{O_{V,\ell}(1)}\exp(-cN^{2\theta-1})=o(1)\,.
\]
The fixed entropy gap gives the same conclusion away from all
uniform laws on subspaces.

If a solution $Z$ lies within a sufficiently small neighbourhood of
the uniform law on a proper subspace $H<V$, then at most $\delta N$
rows of its quotient $x=Z\bmod H$ are nonzero.  The quotient
$e=E\bmod H$ is nonzero because the rows of $E$ affinely span $V$.
Thus \eqref{d2-sparse-quotient}, followed
by a union bound over the finitely many $H$, shows that the equation
$(\widehat A-\omega I)Z=E$ has such a solution with probability
$o(1)$.

On the invertible event the solution is unique.  Summing the row equations gives
$(2-\omega)\one^{\mathsf T}Z=\one^{\mathsf T}E=0$.  Since the rows
of $E$ affinely span $V$,
\[
 \rank[\one,E]=r+1\,.
\]
If in addition $E^{\mathsf T}Z=T$, then $Z$ belongs to the affine fiber
\[
 \{Z:\one^{\mathsf T}Z=0\,,\ E^{\mathsf T}Z=T\}\,,
\]
which has $Q^{N-r-1}=p^{r(N-r-1)}$ points.  Moreover, all but an
$o(1)$ proportion of this fiber satisfy the first condition in
\eqref{d2-local-window}: choose an invertible $(r+1)$-row minor of
$[\one,E]$; the other $N-r-1$ row values are free uniform elements of
$V$, and the chosen rows are then determined.  Since
$y=\omega Z+E$ differs in only $O(1)$ coordinates from $\omega Z$,
\eqref{d2-central-mixing} applies.  The expected number of solutions
in this fiber satisfying \eqref{d2-local-window} is therefore
\[
 Q^{N-r-1}Q^{-(N-1)}+o(1)=Q^{-r}+o(1)=p^{-r^2}+o(1)\,.
\]
Consequently,
\[
 \begin{split}
 \Pp(\text{event in \eqref{finite-field-point}})
 &\leq\E\#\{\text{fiber solutions satisfying
                         \eqref{d2-local-window}}\}\\
 &\quad+\Pp(\text{there is a solution outside
                         \eqref{d2-local-window}})
 \leq p^{-r^2}+o(1)\,,
 \end{split}
\]
which proves \eqref{finite-field-point}.  Summing over the $T$ for
which $T_{12}+T_{21}^{\mathsf T}$ is singular, as above, proves
\eqref{ff-proj}.

It remains to prove \eqref{ff-sing}.  First take
$V=\F_p$ and $E=0$.  Summing the row equations gives
$\sum_i z_i=0$.  The only proper subspace is $\{0\}$, and the sum of
the probabilities over nonzero vectors with at most $\delta N$
nonzero coordinates is $o(1)$ by
\eqref{d2-support-sum}.  The quadratic deficit gives $o(1)$ for
vectors near the uniform law but outside \eqref{d2-local-window},
while the fixed entropy gap gives $o(1)$ away from the neighbourhoods
of $\delta_0$ and the uniform law.  Together with the preceding
sparse-vector estimate, this bounds the contribution of all remaining vectors outside
\eqref{d2-local-window}.  There are
$p^{N-1}(1+o(1))$ vectors satisfying \eqref{d2-local-window} and
having coordinate sum zero, and each has
probability $p^{-(N-1)}(1+o(1))$ by
\eqref{d2-central-mixing}.  This proves
\eqref{finite-field-kernel}.  Every nontrivial kernel contains at
least $p-1$ nonzero vectors, which gives
\eqref{ff-sing}.
\end{proof}

\subsection{Conditioning the configuration model}

\begin{lemma}\label{lem:model-transfer}
Fix $d\geq2$.  If a sequence of matrix events has probability $o(1)$
under either $\mathcal A_{N,d}$ or $\widehat A$, then it has
probability $o(1)$ under the uniform law on $\cM_{N,d}$.  Moreover,
\[
 \Pp_{\cM_{N,d}}(\diag A=0)\longrightarrow e^{-d}\,.
 \eqlabel{diagonal-positive}
\]
\end{lemma}

\begin{proof}
We first prove \eqref{diagonal-positive}.  Use the bipartite
configuration model with
$d$ half-edges at each of the $N$ vertices on each side.  Let
$Z_N^{\rm diag}$ count matched pairs whose two vertices have the same
label, and let $Z_N^{\rm par}$ count unordered pairs of configuration
edges having the same two endpoints.  For fixed nonnegative integers
$k,m$,
direct enumeration of ordered configurations gives
\[
 \E (Z_N^{\rm diag})_k(Z_N^{\rm par})_m
 \longrightarrow d^k\left(\frac{(d-1)^2}{2}\right)^m\,.
\]
Indeed, choices using pairwise disjoint half-edges give the displayed
main term.  Every choice with an overlap identifies at least one extra
vertex or half-edge and is smaller by a factor $O_d(N^{-1})$.  Hence
\[
 (Z_N^{\rm diag},Z_N^{\rm par})\ \Longrightarrow\
 \bigl(\operatorname{Poisson}(d),
       \operatorname{Poisson}((d-1)^2/2)\bigr)\,,
\]
with independent coordinates.  The event $Z_N^{\rm par}=0$ means that
there are no multiple edges, and $Z_N^{\rm diag}=0$ means that there are no diagonal edges.
Since each simple graph has the same number $(d!)^{2N}$ of
configurations,
\[
 \Pp_{\cM_{N,d}}(\diag A=0)
 =\frac{\Pp(Z_N^{\rm diag}=0,Z_N^{\rm par}=0)}
        {\Pp(Z_N^{\rm par}=0)}\longrightarrow e^{-d}\,.
\]

The same calculation gives
$\Pp(Z_N^{\rm par}=0)\to e^{-(d-1)^2/2}>0$.  Conditioning
$\widehat A$ on this event gives the uniform law on
$\cM_{N,d}$.  This proves the transfer from the configuration model
and, when $d=2$, from $\mathcal A_{N,2}$.

Suppose $d\geq3$.  The probability that $P_1,\ldots,P_d$ have
pairwise disjoint matrix supports is bounded away from zero.  For $2\leq k\leq d$,
conditionally on the first $k-1$ permutations with disjoint matrix
supports, the allowed-position matrix for
$P_k$ is $(N-k+1)$-regular, and the van der Waerden bound gives
conditional probability at least $(1-(k-1)/N)^N$.  Conditioning on
disjointness therefore preserves an $o(1)$ probability and gives the
ordered one-factorization law.  The bipartite one-factorization
contiguity theorem of Molloy--Robalewska--Robinson--Wormald
\cite[Corollary~1, Theorem~4, and the final paragraph of
Section~4]{MRRW} transfers this probability bound to the uniform law on $\cM_{N,d}$.
\end{proof}

\subsection{Augmented expansion}

For $S\subset[N]$ and a nonnegative $d$-regular matrix $A$, put
\[
 N_S(i)=\one_S(i)+\sum_{j\in S}A_{ij}\,,\qquad
 V_A^+(S)=\{i:N_S(i)\geq1\}\,,\qquad
 U_A^+(S)=\{i:N_S(i)=1\}\,.
\]
Thus $V_A^+(S)=S\cup\Gamma_A(S)$, where
\[
 \Gamma_A(S)=\{i:A_{ij}>0\text{ for some }j\in S\}\,.
\]

\begin{lemma}\label{lem:expansion}
Fix $d\geq2$ and $0<\eps<d-1$.  There is
$\delta=\delta(d,\eps)>0$ such that, for $A=\widehat A$, with
probability $1-o(1)$, simultaneously for $A$ and
$A^{\mathsf T}$,
\[
 |V_A^+(S)|>(d-\eps)|S|\,,\qquad
 |U_A^+(S)|>(d-1-2\eps)|S|
 \eqlabel{aug-expansion}
\]
for every nonempty $S\subset[N]$ with $|S|\leq\delta N$.  The same
conclusion holds for uniform $A\in\cM_{N,d}$.
\end{lemma}

\begin{proof}
Fix $|S|=s$ and put $m=\lfloor(d-\eps)s\rfloor$.  If
$|V_A^+(S)|\leq m$, there is $T\supset S$, $|T|=m$, containing every
edge entering $S$.  For fixed $S,T$, the probability of this event is
at most
\[
 \frac{(dm)_{ds}}{(dN)_{ds}}\leq\left(\frac mN\right)^{ds}\,.
\]
Hence
\[
 \Pp\bigl(\exists S:|S|=s\,,\ |V_A^+(S)|\leq(d-\eps)s\bigr)
 \leq
 \binom Ns\binom{N-s}{m-s}\left(\frac mN\right)^{ds}
 \leq\biggl(K_{d,\eps}\left(\frac sN\right)^\eps\biggr)^{s}\,.
\]
Choose $\delta$ so that $K_{d,\eps}\delta^\eps<1/4$ and
$(d-\eps)\delta<1$.  Summing first over $s\leq N^{1/2}$ and then over
$N^{1/2}<s\leq\delta N$ gives the first inequality in
\eqref{aug-expansion}.  The same estimate applies to $A^{\mathsf T}$.

Summing $N_S(i)$ over $i$ counts $|S|$ diagonal incidences and $d|S|$
incoming incidences.  If
$u=|U_A^+(S)|$, then
\[
 (d+1)|S|=\sum_iN_S(i)
 \geq u+2(|V_A^+(S)|-u)=2|V_A^+(S)|-u\,,
\]
which gives the second inequality.  Lemma~\ref{lem:model-transfer}
transfers the failure probability to the uniform law on $\cM_{N,d}$.
\end{proof}

\section{Shifted inverse compression and reverse resonance}
\label{sec:compression}

\subsection{Shifted inverse compression}

\begin{proposition}
\label{prop:compression}
Fix $d\geq2$, $\ell\geq1$, and $w\in\Q(i)\setminus\{0,d\}$.  If $A$
is uniform on $\cM_{N,d}$, then, uniformly over the disjoint anchored
sets in \eqref{difference-matrices},
\[
 \Pp\left(
 \begin{array}{c}
 \det(A-wI)\neq0\,,\\[1mm]
 \det\bigl[L_X^{\mathsf T}((A-wI)^{-1}
              +(A-wI)^{-\mathsf T})L_Y\bigr]=0
 \end{array}
 \right)=o(1)\,.
 \eqlabel{compression-zero}
\]
Moreover,
\[
 \Pp(\det(A-wI)=0)=o(1)\,.
 \eqlabel{exact-singularity}
\]
\end{proposition}

\begin{proof}
Write
\[
 w=\frac{u+iv}{q}\,,
 \qquad u,v\in\Z\,,\quad q\in\mathbb N_+\,.
 \eqlabel{rational-shift}
\]
Choose a prime $p\equiv1\pmod4$ which does not divide $q$, and choose
$\iota_p\in\F_p$ with $\iota_p^2=-1$.  Let
\[
 \phi_p:\Z[i,q^{-1}]\longrightarrow\F_p\,,
 \qquad \phi_p(i)=\iota_p\,,\qquad \phi_p(q^{-1})=q^{-1}\,,
\]
be the corresponding ring homomorphism, and set
$w_p=\phi_p(w)=(u+v\iota_p)q^{-1}$.
We exclude the finitely many primes for which
$w_p\in\{0,d\}$.  Such a prime divides one of the two nonzero integers
\[
 u^2+v^2\,,
 \qquad (u-dq)^2+v^2\,.
\]
Dirichlet's theorem gives arbitrarily large remaining primes
congruent to one modulo four.

Reduce $L_X,L_Y$ modulo $p$, retaining the same notation.
Proposition~\ref{prop:ff-comp} gives
\[
 \limsup_{N\to\infty}\Pp\left(
 \begin{array}{c}
 \mathcal A_{N,d}-w_pI\text{ is invertible}\,,\\
 \det[L_X^{\mathsf T}((\mathcal A_{N,d}-w_pI)^{-1}
 +(\mathcal A_{N,d}-w_pI)^{-\mathsf T})L_Y]=0
 \end{array}
 \right)\leq\frac1{p-1}\,.
 \eqlabel{modular-compression}
\]
It also gives
\[
 \limsup_{N\to\infty}
 \Pp(\mathcal A_{N,d}-w_pI\text{ is singular over }\F_p)
 \leq\frac1{p-1}\,.
 \eqlabel{modular-singularity}
\]

For the complex shift, put
\[
 B=q\mathcal A_{N,d}-(u+iv)I\in\Z[i]^{N\times N}\,.
\]
Since $(\mathcal A_{N,d}-wI)^{-1}=qB^{-1}$, the Gaussian integer
\[
 (\det B)^\ell
 \det[L_X^{\mathsf T}(B^{-1}+B^{-\mathsf T})L_Y]
\]
vanishes on the event in \eqref{compression-zero}.  Reducing this
identity modulo one of the remaining split primes shows that
either $\phi_p(B)=q(\mathcal A_{N,d}-w_pI)$ is singular, or the
projected determinant in \eqref{modular-compression} vanishes.  Equations
\eqref{modular-compression}--\eqref{modular-singularity} give
\[
 \limsup_{N\to\infty}\Pp(\text{event in \eqref{compression-zero}})
 \leq\frac2{p-1}\,.
\]
Letting $p$ tend to infinity through these split primes proves
\eqref{compression-zero} in the auxiliary model.  If
$\det(\mathcal A_{N,d}-wI)=0$ over $\C$, then its determinant after clearing denominators
vanishes modulo
every such prime under $\phi_p$.  Equation
\eqref{modular-singularity}, followed
again by $p\to\infty$, proves \eqref{exact-singularity} in this model.

Lemma~\ref{lem:model-transfer} transfers \eqref{compression-zero} and
\eqref{exact-singularity} to the uniform law on $\cM_{N,d}$.  Simultaneous
permutation of the row and column labels acts transitively on the anchored disjoint pairs $(X,Y)$,
proving the asserted uniformity.
\end{proof}

\subsection{Reverse resonance}

For a shifted matrix $C=A-wI$, put $D_N=\binom N2$ and
\[
 \mathcal R_\kappa(C)=
 \begin{cases}
 \#\{i<j:|(e_i-e_j)^{\mathsf T}C^{-1}(e_i-e_j)+1/w|
              \leq e^{-\kappa N}\}\,,&\det C\ne0\,,\\
 D_N\,,&\det C=0\,.
 \end{cases}
\]

\begin{proposition}\label{prop:resonance}
Fix $d\geq2$, $w\in\Q(i)\setminus\{0,d\}$, and $\kappa>0$.  If $A$
is uniform on $\cM_{N,d}$, then
\[
 \zeta_N(\kappa):=\E\frac{\mathcal R_\kappa(A-wI)}{D_N}
 \longrightarrow0\,.
\]
\end{proposition}

\begin{proof}
Fix $0<\alpha<1$.  On the nonsingular event put $G=(A-wI)^{-1}$ and
$q_{ij}=(e_i-e_j)^{\mathsf T}G(e_i-e_j)$, and let
\[
 \mathcal E_{N,\alpha}
 =\{\mathcal R_\kappa(A-wI)\geq\alpha D_N\}\,.
\]
Use the representation \eqref{rational-shift} and put
\[
 B=q(A-wI)=qA-(u+iv)I\in\Z[i]^{N\times N}\,.
\]
There is a constant $\Lambda=\Lambda(d,w)$ such that every row of $B$
has squared Euclidean norm at most $\Lambda$.  Explicitly, one may take
\[
 \Lambda=\max\bigl(dq^2+|u+iv|^2,
 (d-1)q^2+|q-u-iv|^2\bigr)\,.
\]
Hadamard's inequality gives
\[
 |\det B|^2\leq\Lambda^N\,.
 \eqlabel{Hadamard-B}
\]

Choose a fixed $\ell$ so large that
\[
 \ell\kappa>\log\Lambda\,.
 \eqlabel{ell-choice}
\]
If $\mathcal E_{N,\alpha}$ occurs and $A-wI$ is invertible, then the
graph on $[N]$ whose resonant edges satisfy
$|q_{ij}+1/w|\leq e^{-\kappa N}$ has at least $\alpha D_N$ edges.  The
standard supersaturation theorem for $K_{\ell+1,\ell+1}$
\cite{ErdosSimonovits} gives a
constant $\rho=\rho(\alpha,\ell)>0$ such that at least a proportion
$\rho$ of the disjoint anchored pairs $(X,Y)$ in
\eqref{difference-matrices} have every edge between $X$ and $Y$
resonant.

For such a pair, four-point polarization gives, for
$x\in X\setminus\{x_0\}$ and $y\in Y\setminus\{y_0\}$,
\[
 \begin{split}
 \big|(e_x-e_{x_0})^{\mathsf T}(G+G^{\mathsf T})
                 (e_y-e_{y_0})\big|=
 \left|-q_{xy}+q_{xy_0}+q_{x_0y}-q_{x_0y_0}\right|
 \leq4e^{-\kappa N}\,.
 \end{split}
\]
Hadamard's inequality therefore yields
\[
 \left|\det L_X^{\mathsf T}(G+G^{\mathsf T})L_Y\right|
 \leq(4\sqrt\ell\,e^{-\kappa N})^\ell\,.
 \eqlabel{projected-upper}
\]

Since $G=qB^{-1}$,
\[
 G+G^{\mathsf T}
 =qB^{-\mathsf T}(B+B^{\mathsf T})B^{-1}\,.
\]
Two applications of Cauchy--Binet followed by Jacobi's complementary-
minor identity give
\[
 \det L_X^{\mathsf T}(G+G^{\mathsf T})L_Y
 =q^\ell\frac{Z_{X,Y}(B)}{(\det B)^2}\,,
 \qquad Z_{X,Y}(B)\in\Z[i]\,.
 \eqlabel{integer-separation}
\]
Indeed, Cauchy--Binet expands the left side through $\ell$-row minors
of $B^{-1}L_X$ and $B^{-1}L_Y$.  Jacobi's identity writes each such
minor of $B^{-1}$ with the single denominator $\det B$ and a
Gaussian-integer complementary minor.  The two inverse factors thus
produce exactly $(\det B)^2$ in \eqref{integer-separation}; all
remaining factors have entries in $\Z[i]$.
If $Z_{X,Y}(B)\neq0$, then $|Z_{X,Y}(B)|\geq1$, so
\eqref{Hadamard-B} and \eqref{integer-separation} imply
\[
 \left|\det L_X^{\mathsf T}(G+G^{\mathsf T})L_Y\right|
 \geq q^\ell\Lambda^{-N}\,.
 \eqlabel{projected-lower}
\]
By \eqref{ell-choice},
\[
 \frac{q^\ell\Lambda^{-N}}
 {(4\sqrt\ell\,e^{-\kappa N})^\ell}
 =\left(\frac q{4\sqrt\ell}\right)^\ell
 e^{(\ell\kappa-\log\Lambda)N}\longrightarrow\infty\,.
\]
Thus \eqref{projected-upper} and \eqref{projected-lower} are
incompatible for large $N$, and the
projected determinant vanishes for every such anchored pair.

Choose a disjoint anchored pair $(X,Y)$ uniformly and independently of $A$.
Proposition~\ref{prop:compression} and its uniformity give
\[
 \rho\Pp(\mathcal E_{N,\alpha},\det(A-wI)\ne0)
 \leq\Pp\bigl(\det(A-wI)\neq0\,,
 \det L_X^{\mathsf T}(G+G^{\mathsf T})L_Y=0\bigr)+o(1)=o(1)\,.
\]
Together with \eqref{exact-singularity}, this gives
$\Pp(\mathcal E_{N,\alpha})=o(1)$.  Since
$0\leq\mathcal R_\kappa(A-wI)/D_N\leq1$,
\[
 \limsup_{N\to\infty}\zeta_N(\kappa)\leq\alpha\,.
\]
Letting $\alpha\downarrow0$ proves the proposition.
\end{proof}

\section{The small-ball estimate}\label{sec:fixed-degree}

For a subspace $V\subset\C^N$, $P_V$ denotes orthogonal projection
onto $V$.  Since $A$ is real,
\[
 (A-wI)^*=A^{\mathsf T}-\overline w I\,.
\]

For an unordered pair $e=\{i,j\}$, $i<j$, set
\[
 a_e=e_i-e_j\,,
 \qquad P_e=I-a_ea_e^*\,.
\]
Since $\|a_e\|^2=2$, the matrix $P_e$ is the real orthogonal
transposition of coordinates $i,j$.  In particular,
$P_e=P_e^*=P_e^{-1}$ and $P_ea_e=-a_e$.  The map
\[
 T_eA=AP_e
\]
is an involutive bijection of $\cM_{N,d}$.  Choosing $e$ uniformly from
the $D_N=\binom N2$ unordered pairs defines a reversible Markov chain
with uniform stationary measure.

For $w\ne0$, put
\[
 C=A-wI\,,
 \qquad C_e=T_eA-wI=CP_e-wa_ea_e^*\,,
 \qquad C_eP_e=C+wa_ea_e^*\,.
 \eqlabel{two-gauges}
\]

\begin{lemma}\label{lem:cluster}
Fix $d\geq2$ and $w\in\C\setminus\{0\}$.  There is
$c_0=c_0(d,w)>0$ with the following property.  Let $A\in\cM_{N,d}$,
$e\in\binom{[N]}2$, $1\leq r<N$, $\eps\geq0$, and $\eta,t>0$.
Suppose
\[
 \tau_r(C)\leq\eps\,,
 \qquad \tau_{r+1}(C)\geq\eta\,,
\]
and let $U,V$ be the bottom $r$-dimensional left and right singular
subspaces of $C$.  If
\[
 \|P_Ua_e\|\geq t\,,
 \qquad \|P_Va_e\|\geq t\,,
 \qquad \eps\leq c_0\eta t^2\,,
\]
then
\[
 \tau_{r-1}(C_e)\leq\eps\,,
 \qquad \tau_r(C_e)\geq c_0\eta t^2\,.
 \eqlabel{cluster-conclusion}
\]
For $r=1$, the first conclusion is omitted.
\end{lemma}

\begin{proof}
Put $M=CP_e$ and $a=a_e$.  The bottom left singular subspace of $M$
is $U$, while its bottom right singular subspace is $P_eV$.  Since
$P_ea=-a$,
\[
 \|P_{P_eV}a\|=\|P_Va\|\geq t\,.
 \eqlabel{gauge-visible}
\]
Choose a singular-value decomposition of $M$ and replace its bottom
$r$ singular values by zero.  The resulting matrix $M_0$ satisfies
\[
 \|M-M_0\|\leq\eps\,,
 \quad \ker M_0=P_eV\,,
 \quad \ker M_0^*=U\,,
 \quad \|M_0x\|\geq\eta\|x\|\quad(x\perp P_eV)\,.
\]
Put $S=M_0-waa^*$.  Projection of $Sx=0$ onto $U$ gives
$-wP_Ua\,a^*x=0$.  Thus $a^*x=0$ and then $M_0x=0$, so
\[
 \ker S=P_eV\cap a^\perp\,,
 \qquad \dim\ker S=r-1\,.
\]

We estimate $S$ on $(\ker S)^\perp$.  Write
$x=x_V+x_\perp$ according to $P_eV\oplus(P_eV)^\perp$, put $y=Sx$,
and set $\alpha=a^*x$.  Projection onto $U$ and $U^\perp$ gives
\[
 |\alpha|\leq\frac{\|y\|}{|w|t}\,,
 \qquad
 \|x_\perp\|\leq\frac1\eta
 \left(\|y\|+|w|\|a\||\alpha|\right)\,.
\]
Since $x\perp P_eV\cap a^\perp$, the vector $x_V$ is a scalar
multiple of $P_{P_eV}a$.  Hence
\[
 \|x_V\|\leq\frac1t
 \left(|\alpha|+\|a\|\|x_\perp\|\right)\,.
\]
As $\|a\|=\sqrt2$, $\|C\|\leq d+|w|$, and
$\eta\leq\|C\|$, the last two displays imply
\[
 \|x\|\leq C(d,w)\eta^{-1}t^{-2}\|Sx\|
 \qquad(x\perp\ker S)\,.
\]
Thus the first positive singular value of $S$ is at least
$c_1(d,w)\eta t^2$.  Choose $c_0\leq c_1/2$.  Since
\[
 C_e=M-waa^*=S+(M-M_0)\,,
\]
Weyl's inequality proves \eqref{cluster-conclusion}.
\end{proof}

\subsection{Positive visibility}

Small-set expansion and a peeling argument give a polynomial lower
bound on both projection norms for a positive proportion of transpositions.

For every shifted matrix $C$ and every $1\leq r<N$, fix measurable maps
$C\mapsto U_r(C),V_r(C)$ selecting its bottom left and right singular
subspaces.

\begin{proposition}\label{prop:visibility}
Fix $d\geq2$ and $w\in\C\setminus\{0,d\}$.  For uniform
$A\in\cM_{N,d}$, there are constants $T,g>0$ and events $\cG_N$ with
$\Pp(\cG_N^c)=o(1)$ such that the following holds on $\cG_N$.
Put $C=A-wI$ and $t=N^{-T}$.  Simultaneously for every $1\leq r<N$,
if $\tau_r(C)\leq t$, then at least $gD_N$ pairs $e$ satisfy
\[
 \|P_{U_r(C)}a_e\|>t\,,
 \qquad
 \|P_{V_r(C)}a_e\|>t\,.
 \eqlabel{positive-visibility}
\]
We call a pair satisfying \eqref{positive-visibility} visible at rank
$r$.
\end{proposition}

\begin{proof}
Apply Lemma~\ref{lem:expansion} with $\eps=1/4$, and let $\delta_d$
and $\cG_N$ denote the resulting constant and simultaneous event for
$A,A^{\mathsf T}$.  Thus, on $\cG_N$,
\[
 |S\cup\Gamma_A(S)|>\left(d-\frac14\right)|S|
\]
for every nonempty $S\subset[N]$ with $|S|\leq\delta_dN$.

We next obtain a polynomial lower bound on vectors of small support.
Put $L=d+|w|$ and
\[
 \begin{gathered}
 a=\begin{cases}
     \min(1,|w|,|1-w|)\,,&w\ne1\,,\\
     1\,,&w=1\,,
    \end{cases}
 \; \gamma=d-\frac32\,,
 \; \theta=\frac{\gamma}{d+1}\,,
 \lambda=-\log(1-\theta)\,,
 \; \kappa=\frac{\log(1+L/a)}{\lambda}\,.
 \end{gathered}
\]
Every nonzero entry of $C$ has modulus at least $a$.
Fix a nonempty $S$ with $|S|\leq\delta_dN$, and put
\[
 m_C(S)=\#\{(i,j):j\in S\,,\ C_{ij}\ne0\}\,,
 \qquad
 N_C(S)=\{i:C_{ij}\ne0\text{ for some }j\in S\}\,.
\]
If $w\ne1$, then
$m_C(S)\leq(d+1)|S|$ and $N_C(S)=S\cup\Gamma_A(S)$.  If $w=1$, let
$\ell_S$ be the number of adjacency loops based in $S$.  Cancellation
of such loops leaves
\[
 m_C(S)=(d+1)|S|-2\ell_S\,,
 \qquad
 |N_C(S)|>\left(d-\frac14\right)|S|-\ell_S\,.
\]
In both cases the number $n_1(S)$ of rows containing exactly one
nonzero entry in the columns $S$ satisfies
\[
 n_1(S)\geq2|N_C(S)|-m_C(S)
 >\left(d-\frac32\right)|S|=\gamma|S|\,.
\]
Each column meets at most $d+1$ of these rows.  Thus at least
\[
 \frac{n_1(S)}{d+1}>\theta|S|
\]
columns can be assigned distinct rows, each meeting only its assigned
column in $S$.  Remove these columns and repeat.  If $S_k$ denotes the
set remaining after $k$ rounds, then
\[
 |S_k|\leq(1-\theta)^k|S|=e^{-\lambda k}|S|\,.
\]
Thus removing these columns exhausts $S$ in at most
\[
 m\leq1+\left\lceil\frac{\log N}{\lambda}\right\rceil
 \eqlabel{peeling-depth}
\]
rounds.

Let $J_k$ be the columns removed in round $k$ and put
$R_k=J_1\cup\cdots\cup J_k$, with $R_0=\varnothing$.
For $x$ supported on $S$, put $y=Cx$.  The selected rows give
\[
 \|x_{J_k}\|
 \leq a^{-1}\bigl(\|y\|+L\|x_{R_{k-1}}\|\bigr)\,,
 \qquad
 \|x_{R_k}\|\leq(1+L/a)\|x_{R_{k-1}}\|+a^{-1}\|y\|\,.
\]
Using \eqref{peeling-depth}, we obtain
\[
 \|x\|\leq\frac{(1+L/a)^m-1}{L}\|Cx\|
 \leq\frac{(1+L/a)^2}{L}N^\kappa\|Cx\|\,.
 \eqlabel{small-supp-poly}
\]
The same argument applied to $A^{\mathsf T}-\overline wI$ gives this
bound for $C^*$.

Set
\[
 \delta_*=\min\left(\delta_d,\frac1{4(d+1)}\right)\,,
 \qquad t=N^{-(\kappa+2)}\,.
\]
We claim that every unit vector $v$ satisfying $\|Cv\|\leq t$ obeys
\[
 \sup_{z\in\C}\#\{i:|v_i-z|\leq2t\}<(1-\delta_*)N\,.
 \eqlabel{level-deficit}
\]
Otherwise let $S$ be such a level set and put $J=S^c$.  At most
$d|J|$ rows have an $A$-neighbour in $J$, and at most $|J|$ further
row labels lie outside $S$.  Thus at least
$N-(d+1)|J|\geq3N/4$ indices $i\in S$ have all their $A$-neighbours
in $S$.  On these rows,
\[
 |(Cv)_i-(d-w)z|\leq2Lt\,.
\]
Since $w\ne d$, it follows that
\[
 \sqrt{\frac{3N}{4}}\,|d-w|\,|z|
 \leq\|Cv\|+2Lt\sqrt N\,,
 \qquad
 \|v_S\|\leq\sqrt N(2t+|z|)\,,
\]
and therefore
\[
 |z|\leq C_{d,w}\left(t+\frac{\|Cv\|}{\sqrt N}\right)\,,
 \qquad
 \|v_S\|\leq C_{d,w}\bigl(\sqrt Nt+\|Cv\|\bigr)
 =o(N^{-\kappa})\,.
\]
If $J$ is empty this contradicts $\|v\|=1$; otherwise
$\|v_J\|\geq1/2$ for large $N$, while
\[
 \|Cv_J\|\leq\|Cv\|+L\|v_S\|=o(N^{-\kappa})\,,
\]
contradicting \eqref{small-supp-poly}.  This proves
\eqref{level-deficit}, and the same argument applies to $C^*$.

We now find pairs for which both coordinate differences are large.  If unit vectors
$u,v$ satisfy \eqref{level-deficit}, choose $I_1,I_2,I_3$ independently
and uniformly from $[N]$.  Conditionally on $I_1$, the events
\[
 |u_{I_1}-u_{I_2}|>2t\,,
 \qquad |v_{I_1}-v_{I_3}|>2t
\]
have conditional probabilities at least $\delta_*$.  Since $I_2$ and
$I_3$ are conditionally independent, their intersection has conditional
probability at least $\delta_*^2$.  On this intersection, if the pairs
$(I_1,I_2)$ and $(I_1,I_3)$ both fail to
have the two coordinate differences larger than $t$, then
\[
 |u_{I_2}-u_{I_3}|>t\,,
 \qquad |v_{I_2}-v_{I_3}|>t\,.
\]
Thus one of the three pairs always has this property, and
\[
 \delta_*^2N^3
 \leq6N\#\{i<j:|u_i-u_j|>t\,,\ |v_i-v_j|>t\}\,.
\]
Hence at least $(\delta_*^2/3)D_N$ unordered pairs have both
differences larger than $t$.  If
$\tau_r(C)\leq t$, apply this observation to arbitrary unit vectors in
$U_r(C)$ and $V_r(C)$.  Their coordinate differences are bounded by
the corresponding projection norms, so \eqref{positive-visibility}
holds with $T=\kappa+2$ and $g=\delta_*^2/3$.
\end{proof}

\subsection{Proof of the small-ball theorem}

\begin{proof}[Proof of Theorem~\ref{thm:small}]
Fix $c>0$, let $A$ be uniform on $\cM_{N,d}$, and put $C=A-wI$.
Take $T,g$ and $\cG_N$ from Proposition~\ref{prop:visibility}, and put
\[
 \delta_N=\Pp(\cG_N^c)=o(1)\,,\qquad
 t=N^{-T}\,.
\]
On $\cG_N$, simultaneously for every $1\leq k<N$, if
$\tau_k(C)\leq t$,
then at least $gD_N$ labels $e$ satisfy
\[
 \|P_{U_k(C)}a_e\|>t\,,
 \qquad
 \|P_{V_k(C)}a_e\|>t\,.
 \eqlabel{term-rank-vis}
\]

Put $H=cN$.  Write $w=\xi/q$, with
$\xi\in\Z[i]$ and $q\in\mathbb N_+$, and set
\[
 L=d+|w|\,.
\]
Choose a fixed integer $R\geq1$ so large that
\[
 \frac{2\log(qL)}{R+1}<\frac c2\,,
 \eqlabel{terminal-R-choice}
\]
and put
\[
 b=2^{1/R}\,,\qquad
 S=\sum_{j=0}^{R-1}b^j=\frac{b^R-1}{b-1}\,,\qquad
 a_0=\frac1{4S}\,.
\]
For an invertible $C$ satisfying $\tau_1(C)\leq e^{-H}$, define
\[
 x_j=\min(H,(-\log\tau_j(C))_+)\,,
 \qquad 1\leq j\leq R+1\,.
\]
The nonzero Gaussian integer $\det(qC)$ has modulus at least one, so
\[
 q^{-N}\leq\prod_{j=1}^N\tau_j(C)
 \leq\tau_{R+1}(C)^{R+1}L^{N-R-1}\,.
\]
Thus $x_{R+1}\leq N\log(qL)/(R+1)$.  Since $x_1=H$ and $b^R=2$,
\eqref{terminal-R-choice} yields
\[
 H-b^Rx_{R+1}\geq\frac H2\,.
\]
The identity
\[
 H-b^Rx_{R+1}
 =\sum_{r=1}^Rb^{r-1}(x_r-bx_{r+1})
\]
therefore shows that, for some $r\leq R$,
\[
 \Delta_r:=x_r-bx_{r+1}\geq\frac{H}{2S}\,.
\]
The interval
\[
 I_r(C)=\left[x_r-\frac{\Delta_r}{2},x_r\right]
\]
has length at least $a_0H$, is contained in $[a_0H,H]$, and every
$h\in I_r(C)$ satisfies
\[
 \tau_r(C)\leq e^{-h}\,,
 \qquad
 \tau_{r+1}(C)\geq e^{-h/b}\,.
\]
Indeed, $h\leq x_r$, while
$h\geq(x_r+bx_{r+1})/2\geq bx_{r+1}$.  Since $b>1$, the inequalities
$bx_{r+1}\leq h\leq H$ imply $x_{r+1}<H$; hence
$\tau_{r+1}(C)\geq e^{-x_{r+1}}\geq e^{-h/b}$.  Thus, on putting
\[
 E_{r,h}^{(b)}=
 \{\tau_r(C)\leq e^{-h}\,,\ \tau_{r+1}(C)\geq e^{-h/b}\}\,,
\]
we have the pointwise bound
\[
 \begin{split}
 a_0H\,\mathbf1_{\{\det C\ne0\,,\ \tau_1(C)\leq e^{-H}\}}
 \leq
 \int_{a_0H}^{H}\sum_{r=1}^R
       \mathbf1_{E_{r,h}^{(b)}}\,dh\,.
 \end{split}
 \eqlabel{term-cover}
\]

We next prove, uniformly for $1\leq r\leq R$ and
$a_0H\leq h\leq H$, that
\[
 \Pp(E_{r,h}^{(b)})=o(1)\,.
 \eqlabel{term-strata}
\]
Reduce, if necessary, the constant $c_0$ in
Lemma~\ref{lem:cluster} so that $c_0\leq1$, and put
$p_N=c_0t^2$.  Since $R$ is fixed and
\[
 (1-b^{-1})h\geq(1-b^{-1})a_0cN\gg_{c,d,w,R}\log N\,,
\]
for all sufficiently large $N$,
\[
 e^{-h}\leq p_N^R e^{-h/b}\,.
 \eqlabel{term-cluster}
\]

Put $\varepsilon=e^{-h}$ and $\eta=e^{-h/b}$.  Starting with
$A_0=A$ and $C_0=C$, after choosing $e_j$ set
\[
 A_j=T_{e_j}A_{j-1}\,,
 \qquad C_j=A_j-wI\,.
\]
On $E_{r,h}^{(b)}$, if $A_0,\ldots,A_{j-1}\in\cG_N$ and
$e_1,\ldots,e_j$ are chosen visible at the successive ranks
$r,r-1,\ldots,r-j+1$, then
Lemma~\ref{lem:cluster} and \eqref{two-gauges} give
\[
 \tau_{r-j}(C_j)\leq\varepsilon\,,
 \qquad
 \tau_{r-j+1}(C_j)\geq p_N^j\eta\,,
 \qquad 0\leq j\leq r\,,
\]
where the first inequality is omitted for $j=r$.  The condition at
the step from $j$ to $j+1$ is
\[
 \varepsilon\leq p_N^{j+1}\eta\,,
\]
which follows from \eqref{term-cluster} because
$j+1\leq r\leq R$ and $p_N<1$.

Here $\varepsilon<t$ uniformly for $h\geq a_0cN$, so
whenever $A_j\in\cG_N$, \eqref{term-rank-vis} gives at least $gD_N$
choices for the next label; right visibility is preserved by
\eqref{gauge-visible}.  Call an $r$-step path successful if it starts
in $E_{r,h}^{(b)}$ and has these properties.  Every successful path ends at
the invertible shifted matrix $B=C_r$, satisfying
\[
 s_{\min}(B)\geq\gamma_{r,h}\,,\qquad
 \gamma_{r,h}=p_N^r e^{-h/b}\,.
 \eqlabel{term-endpoint}
\]
Put
\[
 \varkappa_0=\frac14(1-b^{-1})a_0c>0\,.
\]

Let $g_N=\lfloor gD_N\rfloor/D_N$, so $g_N\geq g/2$ eventually,
and let $\mu_{r,h}$ be the proportion of successful paths in
$\cM_{N,d}\times\bigl(\binom{[N]}2\bigr)^r$.
We prove
\[
 g_N^r\Pp(E_{r,h}^{(b)})-r\delta_N
 \leq\mu_{r,h}\leq\zeta_N(\varkappa_0)\,.
 \eqlabel{terminal-path-mass}
\]
For the lower bound, put $m=\lfloor gD_N\rfloor$.  Recursively at each
history choose a
set of exactly $m$ next labels: use visible labels satisfying
\eqref{term-rank-vis} when the preceding matrices lie
in $\cG_N$ and satisfy the displayed singular-value bounds, and use
arbitrary labels otherwise.  The selected
family starting from $E_{r,h}^{(b)}$ has normalized mass
\[
 \frac{\#\{\text{selected paths starting in }E_{r,h}^{(b)}\}}
 {|\cM_{N,d}|D_N^r}
 =g_N^r\Pp(E_{r,h}^{(b)})\,.
\]
For every $0\leq j<r$,
\[
 \frac{\#\{(A,e_1,\ldots,e_r):A_j\notin\cG_N\}}
 {|\cM_{N,d}|D_N^r}=\delta_N\,,
\]
because $A_j$ is uniform on $\cM_{N,d}$.  Every selected path with
$A_0,\ldots,A_{r-1}\in\cG_N$ is successful.  For each $j$, the
selected paths with $A_j\notin\cG_N$ form a subset of the event counted
in the last display.  A union bound proves the left
inequality in \eqref{terminal-path-mass}.

For the upper bound, we use the right visibility of the last label.  Let $e=e_r$
be the last label, put $a=a_e$, and choose a unit vector $v$ spanning
the bottom right singular subspace of $C_{r-1}$.  By construction,
\[
 \|C_{r-1}v\|\leq e^{-h}\,,\qquad |a^*v|>t\,.
\]
With $P_e=I-aa^*$, set
\[
 M=C_{r-1}P_e=B+waa^*\,,\qquad x=P_ev\,.
\]
Then
\[
 \|Mx\|=\|C_{r-1}v\|\leq e^{-h}\,,
 \qquad |a^*x|=|a^*v|>t\,.
\]
Put $s=a^*x$ and $r_0=Mx$.  Since $Bx=-was+r_0$,
\[
 (1+wa^*B^{-1}a)s=a^*B^{-1}r_0\,.
\]
Using $|s|>t$, $\|a\|=\sqrt2$, and
\eqref{term-endpoint} gives
\[
 |1+wa^*B^{-1}a|
 \leq\sqrt2\,p_N^{-r}t^{-1}e^{-(1-b^{-1})h}\,.
 \eqlabel{term-inh-res}
\]
The prefactor is polynomial in $N$, uniformly for $r\leq R$.  Hence
the last forward label satisfies
\[
 |a^*B^{-1}a+1/w|\leq e^{-\varkappa_0N}
 \eqlabel{term-res-label}
\]
for every sufficiently large $N$, uniformly in $r,h$.
Indeed, the left side is
$|w|^{-1}|1+wa^*B^{-1}a|$, and the fixed factor $|w|^{-1}$ is absorbed
by the exponential decay in \eqref{term-inh-res}.

Reverse the labelled path.  For a fixed reverse starting point $B$,
its first label has at most $\mathcal R_{\varkappa_0}(B)$ choices by
\eqref{term-res-label}; the remaining $r-1$ labels have at
most $D_N^{r-1}$ choices.  Path reversal is a bijection, including for
repeated labels, and therefore
\[
 \mu_{r,h}
 \leq\frac1{|\cM_{N,d}|D_N^r}
 \sum_{A'\in\cM_{N,d}}\mathcal R_{\varkappa_0}(A'-wI)D_N^{r-1}
 =\zeta_N(\varkappa_0)=o(1)\,.
\]
The last equality is Proposition~\ref{prop:resonance}.
This proves the right inequality in \eqref{terminal-path-mass}.  Hence
\[
 \sup_{\substack{1\leq r\leq R\\a_0H\leq h\leq H}}
 \Pp(E_{r,h}^{(b)})
 \leq\left(\frac2g\right)^R
       \bigl(\zeta_N(\varkappa_0)+R\delta_N\bigr)=o(1)\,,
\]
which proves \eqref{term-strata}.

Taking expectations in \eqref{term-cover} and applying
Fubini now gives
\[
 \begin{split}
 a_0H\,\Pp(\det C\ne0\,,\ \tau_1(C)\leq e^{-H})
 &\leq HR
 \sup_{\substack{1\leq r\leq R\\a_0H\leq h\leq H}}
 \Pp(E_{r,h}^{(b)})=o(H)\,.
 \end{split}
\]
Equation \eqref{exact-singularity} proves
the theorem in the loops-allowed model.  Finally,
Lemma~\ref{lem:model-transfer} shows that $\diag A=0$ has
probability bounded away from zero, so conditioning transfers the estimate to
$\cM^0_{N,d}$.
\end{proof}

\section{A polynomial hard-edge count}\label{sec:count}
We use the following theorem in Section~\ref{sec:girko}
to remove the logarithmic singularity at zero.

\begin{theorem}
\label{thm:polynomial-count}
Fix an integer $d\geq2$ and
$w\in\Q(i)\setminus\{0,d\}$.  There is a finite
$K=K(d,w)$ such that, if $A$ is uniform on either $\cM_{N,d}$ or
$\cM^0_{N,d}$, then
\[
 n_{N^{-K}}(A-wI)=O_{\Pp}(1)\,.
\]
\end{theorem}

Throughout this section, fix
\[
 d\geq2\,,
 \qquad
 w=\frac \xi q\in\Q(i)\setminus\{0,d\}\,,
 \qquad
 q\in\mathbb N_+\,,\quad \xi\in\Z[i]\setminus\{0\}\,,
\]
and put
$
 C=A-wI.
$
Choose
\[
 0<\eps<\frac{d-1}{2}\,,
 \qquad
 0<a<\min\left(\frac12,\frac1{d-1+|w|},\frac{|w|}{d}\right)\,,
\]
and define
\[
 \gamma=\frac{2d-2\eps}{d+1}>1\,,
 \qquad
 \kappa=\frac{\log(1/a)}{\log\gamma}\,.
\]
Fix
\[
 K_{\rm H}>\frac12+\kappa\,,
\]
let $\delta=\delta(d,\eps)$ be supplied by
Lemma~\ref{lem:expansion}, and fix $0<\alpha<1/4$.  These choices
depend on $w$ only through $|w|$ and are therefore the same for $w$
and $\overline w$.

\begin{theorem}\label{thm:HGALR}
For every fixed $c_0>0$, if $A$ is uniform on $\cM_{N,d}$, then with
probability $1-o(1)$ there is no vector $z\in\C^N$ satisfying
\[
 \|Cz\|_\infty\leq1\,,
 \qquad
 \|z\|_2\geq c_0N^{K_{\rm H}}\,,
 \qquad
 \max_{b\in\C}\#\{j:z_j=b\}\geq\alpha N\,.
 \eqlabel{level-vector}
\]
\end{theorem}

We first prove the theorem for the configuration matrix
$A=\widehat A$.  For each row $i$, let
$\sigma_1(i),\ldots,\sigma_d(i)$ be the column labels incident to its
labelled half-edges, retaining repetitions.  The row
equation is
\[
 \sum_{j=1}^dz_{\sigma_j(i)}-wz_i=(Cz)_i\,.
 \eqlabel{hardedge-row}
\]

\subsection{Limiting row laws}

In this representation,
\[
 N_S(i)=\one_S(i)+\sum_{j=1}^d\one_S(\sigma_j(i))\,.
\]
We abbreviate $U_A^+(S)$ by $U^+(S)$.

\begin{lemma}
\label{lem:superlevel}
Assume the event in Lemma~\ref{lem:expansion}.  Let $z\in\C^N$ satisfy
$\|Cz\|_\infty\leq1$, and put
\[
 S_t:=\{i:|z_i|>t\}\,.
\]
There is a fixed $t_*=t_*(d,w,a)$ such that
\[
 |S_{at}|>\gamma|S_t|
\eqlabel{superlevel-growth}
\]
whenever $t\geq t_*$ and $1\leq|S_t|\leq\delta N$.
\end{lemma}

\begin{proof}
Put
\[
 t_*=\max\left(\frac1{1-a(d-1+|w|)},
                    \frac1{|w|-da}\right)\,,
\]
whose denominators are positive by the choice of $a$.  Put $S=S_t$
and $T=S_{at}$.  Fix $i\in U^+(S)$.  If its unique occurrence from
$S$ is an input and none of the other $d$ occurrences belongs to
$T\setminus S$, then \eqref{hardedge-row} gives
\[
 |z_{\sigma_j(i)}|
 \leq 1+(d-1+|w|)at\leq t\,,
\]
which is impossible.  If its unique occurrence from $S$ is the root
and none of the $d$ inputs belongs to $T\setminus S$, then
\[
 |w||z_i|\leq1+dat\leq|w|t\,,
\]
which is again impossible.  Thus every $i\in U^+(S)$ has another
augmented occurrence in $T\setminus S$.

Each coordinate in $T\setminus S$ occurs once as a root and $d$ times
as an input, with multiplicity.  Therefore
\[
 |U^+(S)|\leq(d+1)|T\setminus S|\,.
\]
Combining this with \eqref{aug-expansion} gives
\[
 |T|>|S|+\frac{d-1-2\eps}{d+1}|S|=\gamma|S|\,,
\]
as required.
\end{proof}

\begin{proposition}
\label{prop:canonical}
Fix $c_0>0$.  Suppose that, along a sequence $N\to\infty$, the event
in Lemma~\ref{lem:expansion} holds and there are vectors $z=z_N$ satisfying
\eqref{level-vector}.  Let $\mathcal S_d$ denote the permutations of
$[d]$.  Then there are scales $\tau_N\to\infty$ such that the measures
\[
 \mu_N:=\frac1N\sum_{i=1}^N\delta_{z_i/\tau_N}\,,
\]
and
\[
 \pi_N:=\frac1{d!N}\sum_{\rho\in\mathcal S_d}\sum_{i=1}^N
 \delta_{(z_i/\tau_N,z_{\sigma_{\rho(1)}(i)}/\tau_N,\ldots,
                    z_{\sigma_{\rho(d)}(i)}/\tau_N)}
\]
converge weakly after passage to a subsequence.  Their limits $\mu$ and
$\pi=\mathcal L(X,Y_1,\ldots,Y_d)$ have the
following properties:
\begin{enumerate}
\item every marginal of $\pi$ is $\mu$;
\item $\sum_{j=1}^dY_j=wX$ almost surely;
\item $\mu$ is nonconstant and has an atom of mass at least $\alpha$;
\item for some fixed $C_{\rm tail},\beta>0$,
\[
 \sup_N\mu_N\{|x|>R\}\leq C_{\rm tail}R^{-\beta}\,,
 \qquad R\geq1;
\eqlabel{canonical-tail}
\]
\end{enumerate}
\end{proposition}

\begin{proof}
Let $M_N=\|z_N\|_\infty$.  The norm assumption in
\eqref{level-vector} gives
\[
 M_N\geq c_0N^{K_{\rm H}-1/2}\,.
\]
Set $t_0=M_N/2$, $t_j=a^jt_0$, and let $J$ be the first index such that
$|S_{t_J}|>\delta N$.  Such a $J$ exists for all large $N$.  Otherwise
Lemma~\ref{lem:superlevel} could be iterated from a nonempty $S_{t_0}$
until the level reached the fixed cutoff $t_*$, giving more than $N$
coordinates because
\[
 (K_{\rm H}-1/2)\frac{\log\gamma}{\log(1/a)}>1\,.
\]
Put $\tau_N=t_J$.  Iterating \eqref{superlevel-growth} up to this index gives
$J\leq\log_{\gamma}(\delta N)+O(1)$ and hence
\[
 \tau_N\geq cM_NN^{-\kappa}
 \geq cN^{K_{\rm H}-1/2-\kappa}\longrightarrow\infty\,.
\]
Minimality and backward iteration give, for every $m\geq1$,
\[
 \frac1N|S_{\tau_Na^{-m}}|\leq
 \delta\gamma^{1-m}\,.
\eqlabel{canonical-shell}
\]
For $m>J$, the set is empty after at most one further step because
$a<1/2$.  Thus \eqref{canonical-shell} is uniform, and interpolation between its levels
gives \eqref{canonical-tail}, with
$\beta=\log\gamma/\log(1/a)>0$.

The definition of this index gives
\[
 \mu_N\{|x|>1\}>\delta\,.
\eqlabel{canonical-mass}
\]
The tail bound makes $(\mu_N)$ tight.  Each coordinate occurs once as
a root and $d$ times among all input slots.  Symmetrization therefore
makes every marginal of $\pi_N$ equal to $\mu_N$, so $(\pi_N)$ is
tight as well.  Extract a common weakly
convergent subsequence.  By applying the closed-set part of Portmanteau
to $\{|x|\geq1\}$, \eqref{canonical-mass} gives
$\mu\{|x|\geq1\}\geq\delta$, so $\mu$ is not concentrated at zero.

Choose $c_N\in\C$ such that $\#\{j:z_j=c_N\}\geq\alpha N$.
Equation \eqref{canonical-shell} shows that $c_N/\tau_N$ remains bounded.  Pass to a further
subsequence on which it converges to $c_\infty$.  Applying Portmanteau to closed disks
and then letting their radii tend to zero gives
\[
 \mu(\{c_\infty\})\geq\alpha\,.
\]
In particular, $\mu$ cannot be a point mass away from zero: if it were
$\delta_c$, then the row equation in item (ii) would give $dc=wc$,
and hence $c=0$ because $w\neq d$, contradicting \eqref{canonical-mass}.

Finally, all marginals of $\pi$ are $\mu$.  Since
$\|Cz\|_\infty/\tau_N\to0$, every limit row satisfies
$\sum_jY_j=wX$.
\end{proof}

\subsection{A finite partition with negative pressure}

For a finite $\mathcal C\subset\C$, define
\[
 \begin{split}
 \mathcal F(\mathcal C):=\mathcal C
 &\cup\left\{w^{-1}(c_1+\cdots+c_d):c_j\in\mathcal C\right\}\\
 &\cup\bigg\{wc'-\sum_{\ell\neq j}c_\ell:
       c',c_\ell\in\mathcal C\,,\ 1\leq j\leq d\bigg\}\,.
 \end{split}
\]

\begin{lemma}
\label{lem:limit-closure}
Let $(\mu_N,\pi_N)\Rightarrow(\mu,\pi)$ be a limit supplied by
Proposition~\ref{prop:canonical}, and put
$\eta=d-1-2\eps>0$ and $\theta=\eta/(d+1)$.  If
$\mathcal C\subset\C$ is finite and $0<\mu(\mathcal C^c)<\delta$, then
\[
 \mu(\mathcal F(\mathcal C)\setminus\mathcal C)
 \geq\theta\mu(\mathcal C^c)\,.
\eqlabel{closure-growth}
\]
\end{lemma}

\begin{proof}
For $r>0$, let $\mathcal C^{(r)}=\bigcup_{c\in\mathcal C}B(c,r)$.
Choose $r_k\downarrow0$ so that $\mu(\partial\mathcal C^{(r_k)})=0$.
For each sufficiently small fixed $r=r_k$, the set
\[
 S_{N,r}:=\{i:z_i/\tau_N\notin\mathcal C^{(r)}\}
\]
has size between $1$ and $\delta N$ for all large $N$.  Applying \eqref{aug-expansion}
and using that every coordinate boundary has $\pi$-mass zero gives
\[
 \pi\{\text{exactly one of }X,Y_1,\ldots,Y_d
                \text{ lies outside }\mathcal C^{(r)}\}
 \geq\eta\mu((\mathcal C^{(r)})^c)\,.
\]
Dominated convergence on the left and continuity from below on the right
as $r_k\downarrow0$ yield the same inequality with $\mathcal C$ in place of
$\mathcal C^{(r)}$.  By the row equation, the unique outside value then
belongs to $\mathcal F(\mathcal C)\setminus\mathcal C$.  Since each of the $d+1$
coordinates has marginal $\mu$, a union bound over its possible position
gives
\[
 \eta\mu(\mathcal C^c)
 \leq(d+1)\mu(\mathcal F(\mathcal C)\setminus\mathcal C)\,,
\]
which is \eqref{closure-growth}.
\end{proof}

For a finite Borel partition $\mathcal P$ of $\C$, define its
pressure under $\pi$ by
\[
\chi_{\mathcal P}(\pi)
 :=H(\mathcal P(X),\mathcal P(Y_1),\ldots,\mathcal P(Y_d))
   -dH(\mathcal P(X))\,.
\]
Throughout this section, $H$ denotes discrete Shannon entropy, with natural
logarithms, and $D(\cdot\Vert\cdot)$ denotes relative entropy.

\begin{proposition}
\label{prop:negative-quotient}
Every limit in Proposition~\ref{prop:canonical} admits a finite Borel partition
$\mathcal P$, whose cells may be chosen to have $\mu$-null boundaries,
such that
\[
 \chi_{\mathcal P}(\pi)<0\,.
\eqlabel{negative-pressure}
\]
\end{proposition}

\begin{proof}
Let $\varrho=\mathcal L(Y_1,\ldots,Y_d)$.  We separate the purely atomic
and mixed cases.

\smallskip
\noindent\emph{Purely atomic marginal.}
Choose a finite set $\mathcal C_0$ with $\mu(\mathcal C_0^c)<\delta$ and put
$\mathcal C_{n+1}=\mathcal F(\mathcal C_n)$.  If
$\mu(\mathcal C_0^c)=0$, then $\mu$ has finite
support.  Otherwise, applying Lemma~\ref{lem:limit-closure} at every
stage for which the tail is positive gives
\[
 \mu(\mathcal C_n^c)\leq\delta(1-\theta)^n\,.
\eqlabel{closure-tail}
\]
Write $\mathcal C_0=\{c_1,\ldots,c_m\}$ and fix
\[
 L>\max(1,d/|w|,|w|+d-1,|\xi q|)\,.
\]
Every
$z\in\mathcal C_n$ has a representation
\[
 z=\sum_{r=1}^m\gamma_rc_r\,,
 \qquad
 \gamma_r\in(\xi q)^{-n}\Z[i]\,,
 \qquad
 \sum_r|\gamma_r|\leq L^n\,.
\eqlabel{closure-coeff}
\]
This is clear for $n=0$.  For the induction step, write
\[
 c^{(j)}=\sum_{r=1}^m
 \frac{\alpha_{j,r}}{(\xi q)^n}c_r\,,
 \qquad \alpha_{j,r}\in\Z[i]\,.
\]
After writing the coefficients with denominator $(\xi q)^{n+1}$, the numerator of the
$r$th coefficient is respectively
\[
 \xi q\alpha_{0,r}\,,
 \qquad q^2\sum_{j=1}^d\alpha_{j,r}\,,
 \qquad \xi^2\alpha_{0,r}-\xi q\sum_{\ell\neq j}\alpha_{\ell,r}
\]
for an old element of $\mathcal C_n$, for
$w^{-1}\sum_{j=1}^dc^{(j)}$, and for
$wc^{(0)}-\sum_{\ell\neq j}c^{(\ell)}$.  These are Gaussian integers.
This proves the assertion about the denominators.  The
coefficient $\ell^1$-norm is multiplied by at most
$\max(1,d/|w|,|w|+d-1)<L$, which proves \eqref{closure-coeff} by
induction.  Moreover,
\[
 \#\{\gamma\in(\xi q)^{-n}\Z[i]:|\gamma|\leq L^n\}
 \leq C(1+(L|\xi q|)^n)^2=e^{O(n)}\,.
\]
Since $m$ is fixed, counting the coefficient vectors gives
\[
 |\mathcal C_n|\leq e^{K_0(n+1)}
\eqlabel{closure-size}
\]
for a fixed $K_0$.

Let $T(x)$ be the least $n\geq0$ for which $x\in\mathcal C_n$, with
an arbitrary value on the null set outside $\bigcup_n\mathcal C_n$.
Since $\mu(\mathcal C_n^c)\to0$, $T(X)$ is finite almost
surely.  Equations \eqref{closure-tail}--\eqref{closure-size} show that
$T(X)$ has a geometric tail, and therefore finite entropy and finite
mean.  Conditional on $T(X)=n$, the variable $X$ takes at most $|\mathcal C_n|$
values, so
\[
 H(X\mid T(X))\leq K_0\E(T(X)+1)<\infty\,.
\]
Since $T(X)$ is a function of $X$,
\[
 H(\mu)=H(X)=H(T(X))+H(X\mid T(X))<\infty\,.
\]
The row equation expresses $X$ as a function of the
input tuple, so
\[
 H(X,Y_1,\ldots,Y_d)-dH(\mu)
 =-D(\varrho\Vert\mu^{\otimes d})\,.
\]
This number is strictly negative.  Equality would make the inputs iid
with law $\mu$.  Thus independent variables $X_1,\ldots,X_d$ with law
$\mu$ would satisfy
\[
 X_1+\cdots+X_d\stackrel{\rm d}=wX\,.
\]
This is impossible for every $w\neq0$.  Put
$S=X_1+\cdots+X_d$.  Since multiplication by $w$ is injective,
$H(S)=H(X)$, whereas
\[
 H(S\mid X_2,\ldots,X_d)=H(X_1)=H(X)\,.
\]
Thus $S$ is independent of $(X_2,\ldots,X_d)$, and in particular of
$X_2$.  Write $S=X_2+T$, where $T$ is independent of $X_2$, and
identify $\C$ with $\R^2$.  If $\phi_X$ and $\phi_T$ are the
corresponding characteristic functions, independence gives
\[
 \phi_T(u)\phi_X(u+v)
 =\phi_T(u)\phi_X(u)\phi_X(v)\,,
 \qquad u,v\in\R^2\,.
\]
For $u$ near zero, $\phi_T(u)\neq0$, so taking $v=-u$ yields
$|\phi_X(u)|=1$.  If $X'$ is an independent copy, then
$e^{\mathrm i\langle u,X-X'\rangle}=1$ almost surely for every $u$ in
a countable dense subset of a neighbourhood of zero.  Continuity in
$u$ forces $X=X'$ almost surely.  This contradicts the fact that
$\mu$ is nonconstant.

Let
\[
 \mathcal Q_n=\bigl\{\{z\}:z\in\mathcal C_n\bigr\}
              \cup\{\mathcal C_n^c\}\,.
\]
These finite partitions are nested and generate the point partition of
the countable atomic support modulo a null set.  Since all the relevant
countable entropies are finite, monotone convergence of entropy gives
\[
 \begin{split}
 H(\mathcal Q_n(X))\nearrow H(X)\,,\quad 
 H(\mathcal Q_n(X),\mathcal Q_n(Y_1),\ldots,
                   \mathcal Q_n(Y_d))\nearrow H(X,Y_1,\ldots,Y_d)\,.
 \end{split}
\]
Consequently,
\[
 \chi_{\mathcal Q_n}(\pi)
 \longrightarrow-D(\varrho\Vert\mu^{\otimes d})<0\,.
\]
Choose an $n$ for which this pressure is negative.  Replace the finitely
many singleton cells by pairwise disjoint disks, with radii outside the
countable set for which a boundary circle has positive $\mu$-mass.  As
the radii decrease, the total mass added by the disks tends to zero;
hence the induced joint finite law converges in total variation to the
singleton law.  Continuity of finite-alphabet entropy preserves the
negativity of the pressure.

\smallskip
\noindent\emph{Mixed atomic--nonatomic marginal.}
Let $\mathcal A$ be the countable atom set and
$s=\mu(\mathcal A)\in(0,1)$.  Suppose first that
\[
 \pi\{X\in\mathcal A\,,\ \text{some }Y_j\notin\mathcal A\}>0\,.
\eqlabel{mixed-event}
\]
Then $\varrho$ has a component singular with respect to $\mu^{\otimes d}$.
On \eqref{mixed-event}, the inputs lie on one of the countably many hyperplanes
\[
 y_1+\cdots+y_d=wa\,,
 \qquad a\in\mathcal A\,,
\]
and at least one coordinate belongs to the nonatomic part.  Fubini's
theorem shows that $\mu^{\otimes d}$ assigns zero mass to this set.
Thus $D(\varrho\Vert\mu^{\otimes d})=\infty$.

Identify $\C$ with $\R^2$.  Choose
$\zeta=(\zeta_1,\zeta_2)\in[0,1)^2$ such that
\[
 \mu\{\re z=\zeta_1+k2^{-n}\}
 =\mu\{\im z=\zeta_2+k2^{-n}\}=0
\]
for every $n\geq1$ and $k\in\Z$.  The excluded set of translates is
countable in each coordinate.  Put
\[
 R_n=[\zeta_1-2^n,\zeta_1+2^n)
     \times[\zeta_2-2^n,\zeta_2+2^n)\,,
\]
and let $\mathcal P_n$ consist of all half-open squares of mesh $2^{-n}$
from the translated grid with boundary lines
$\re z=\zeta_1+k2^{-n}$ and $\im z=\zeta_2+k2^{-n}$ that are
contained in $R_n$, together with the overflow cell $\C\setminus R_n$.
Then $(\mathcal P_n)$ is nested, every cell is a $\mu$-continuity set,
the partitions generate the Borel sigma-field, and
$\log|\mathcal P_n|=O(n)$.  Let $p_n$ be the law of
$\mathcal P_n(X)$, and let $\varrho_n$ be the joint law of
$\mathcal P_n(Y_1),\ldots,\mathcal P_n(Y_d)$.  The monotone
approximation theorem for
relative entropy along nested generating finite partitions gives
\[
 D(\varrho_n\Vert p_n^{\otimes d})
 \nearrow D(\varrho\Vert\mu^{\otimes d})=\infty\,.
\eqlabel{coarse-KL}
\]
Let $O_n=\C\setminus R_n$ and
$E_n=\bigcap_{j=1}^d\{Y_j\notin O_n\}$.  Then
$E_n\in\sigma(\mathcal P_n(Y_1),\ldots,\mathcal P_n(Y_d))$.  On $E_n$,
once the $d$ input cells are fixed,
the identity $X=w^{-1}\sum_jY_j$ confines $X$ to a set of
diameter $O_{d,w}(2^{-n})$.  Hence its possible $\mathcal P_n$-labels
number at most a fixed $M(d,w)$.  On $E_n^c$ there are at most
$|\mathcal P_n|$ possible labels.  Moreover,
\[
 \Pp(E_n^c)\leq d\mu(O_n)=O(2^{-\beta n})
\]
by \eqref{canonical-tail}.  Thus
\[
 \begin{split}
 &H(\mathcal P_n(X)\mid\mathcal P_n(Y_1),\ldots,
                     \mathcal P_n(Y_d))\\
 &\qquad\leq\Pp(E_n)\log M(d,w)
       +\Pp(E_n^c)\log|\mathcal P_n|\leq\log M(d,w)+O(n2^{-\beta n})\,.
 \end{split}
\eqlabel{output-conditional}
\]
Since every marginal of $\varrho_n$ is $p_n$,
\[
 D(\varrho_n\Vert p_n^{\otimes d})=dH(p_n)-H(\varrho_n)\,.
\]
The entropy chain rule therefore gives
\[
 \chi_{\mathcal P_n}(\pi)
 =H(\mathcal P_n(X)\mid\mathcal P_n(Y_1),\ldots,
                         \mathcal P_n(Y_d))
  -D(\varrho_n\Vert p_n^{\otimes d})\,.
\]
Together with \eqref{coarse-KL} and \eqref{output-conditional}, this
proves \eqref{negative-pressure} for all large $n$.

It remains to consider the case where \eqref{mixed-event} has probability zero.  Then
\[
 \{X\in\mathcal A\}\subseteq
 \bigcap_{j=1}^d\{Y_j\in\mathcal A\}
 \quad\pi\text{-almost surely}\,.
\]
The first event has probability $s$, while the second has probability at
most $\Pp(Y_1\in\mathcal A)=s$.  Equality follows, and all $d+1$ atom
indicators coincide.  The common binary partition
$\{\mathcal A,\mathcal A^c\}$ has pressure
\[
 -(d-1)h(s)<0\,,
\]
where $h(s)=-s\log s-(1-s)\log(1-s)$ is binary entropy.  For
$\varepsilon>0$, choose a finite
$\mathcal A_\varepsilon\subset\mathcal A$ such that
$\mu(\mathcal A\setminus\mathcal A_\varepsilon)<\varepsilon$.
Around its points choose pairwise disjoint disks $B(a,r_a)$ with
$\mu(\partial B(a,r_a))=0$ and
\[
 \sum_{a\in\mathcal A_\varepsilon}
 \mu\bigl(B(a,r_a)\setminus\{a\}\bigr)<\varepsilon\,.
\]
For $D_\varepsilon=\bigcup_{a\in\mathcal A_\varepsilon}B(a,r_a)$,
$\mu(\mathcal A\mathbin{\triangle}D_\varepsilon)<2\varepsilon$.
Since all $d+1$ marginals of $\pi$ equal $\mu$, the joint binary laws
induced by $\{\mathcal A,\mathcal A^c\}$ and
$\{D_\varepsilon,D_\varepsilon^c\}$ differ in total variation by at
most $2(d+1)\varepsilon$.  Continuity of entropy on a finite alphabet shows that
the latter pressure tends to $-(d-1)h(s)$ as $\varepsilon\downarrow0$.
It is therefore negative for small $\varepsilon$, and both cells have
$\mu$-null boundary.
\end{proof}

\subsection{A finite-colouring bound}

For a colouring $c:[N]\to[m]$, let $p_c$ be its empirical colour law and
let $\rho_c$ be the empirical law of
\[
 (c(i),c(\sigma_1(i)),\ldots,c(\sigma_d(i)))\,.
\]
Put
\[
 \chi(c)=H(\rho_c)-dH(p_c)\,.
\]

\begin{lemma}\label{lem:colouring}
For every $m\geq2$ and $\zeta>0$,
\[
 \Pp\{\exists c:[N]\to[m]:\chi(c)\leq-\zeta\}
 \leq(dN+1)^{C_dm^{d+1}}e^{-\zeta N}\,.
\]
\end{lemma}

\begin{proof}
Fix empirical types $p$ and $\rho$.  Feasibility requires that the
root marginal of $\rho$ is $p$ and that the sum of its $d$ input
marginals is $dp$.  The union bound over all root colourings of type
$p$ and all input-slot colour words of row type $\rho$ is
\[
 \frac{N!}{\prod_a(Np_a)!}
 \frac{\prod_a(Np_a)!}{\prod_{\boldsymbol a}(N\rho_{\boldsymbol a})!}
 \frac{\prod_a(dNp_a)!}{(dN)!}\,.
 \eqlabel{colour-type-count}
\]
The three factors choose the root word, assign the input tuples, and
give the probability that the uniform matching produces the specified
input-slot colours.  Stirling's formula bounds
\eqref{colour-type-count} by
\[
 \exp\bigl(N[H(\rho)-dH(p)]
       +O_d(m^{d+1}\log(dN+1))\bigr)\,.
\]
There are at most $(dN+1)^{m+m^{d+1}}$ feasible pairs of types.
Summing over those with pressure at most $-\zeta$ proves the lemma.
\end{proof}

\begin{proof}[Proof of Theorem~\ref{thm:HGALR}]
Let
\[
 k_N=\lfloor\log\log N\rfloor\,.
\]
Lemma~\ref{lem:colouring}, applied for every $m\leq k_N$ with
$\zeta=k_N^{-1}$, shows that with probability $1-o(1)$,
\[
 \chi(c)>-k_N^{-1}
 \quad\text{for every colouring with at most }k_N\text{ colours}\,.
 \eqlabel{colouring-event}
\]
Indeed, the total exceptional probability is at most
\[
 \sum_{m=1}^{k_N}
 \exp\bigl(-N/k_N+C_dk_N^{d+1}\log(dN+1)\bigr)=o(1)\,.
\]
Suppose that the conclusion of Theorem~\ref{thm:HGALR} fails with
probability bounded away from zero in the configuration model.
Intersect this event with \eqref{colouring-event} and the event in
Lemma~\ref{lem:expansion}, and choose matrices and vectors along a
subsequence.  Proposition~\ref{prop:canonical} gives a further
subsequence with limit $(\mu,\pi)$.
Proposition~\ref{prop:negative-quotient} supplies a finite Borel
partition $\mathcal P$ with $m_0$ cells, $\mu$-null boundaries, and
pressure $-\gamma_0<0$.

Colour the normalized coordinates by $\mathcal P$, denote the
resulting colouring by $c_N$, and put $\rho_N=\rho_{c_N}$.
The partition law of the symmetrized row measure is
\[
 \overline\rho_N=\frac1{d!}\sum_{\sigma\in\mathcal S_d}\rho_N^\sigma\,,
\]
where $\rho_N^\sigma$ permutes the input coordinates.  Because the
cells of $\mathcal P$ have $\mu$-null boundaries, weak convergence of
$\pi_N$ and $\mu_N$ gives
\[
 H(\overline\rho_N)-dH(p_{c_N})\longrightarrow-\gamma_0\,.
\]
By concavity of entropy,
\[
 \chi(c_N)=H(\rho_N)-dH(p_{c_N})
 \leq H(\overline\rho_N)-dH(p_{c_N})\,.
\]
For all sufficiently large $N$,
\[
 m_0\leq k_N\,,\qquad
 \chi(c_N)<-\gamma_0/2<-k_N^{-1}\,,
\]
contrary to \eqref{colouring-event}.  This proves the theorem for
$\widehat A$.  Lemma~\ref{lem:model-transfer} transfers the
estimate to the uniform law on $\cM_{N,d}$.
\end{proof}

\subsection{Rectangular invertibility}

For $J\subset[N]$, let $E_J:\C^J\to\C^N$ be the coordinate embedding and
define
\[
 \mathcal B_J(C):=
 \left[CE_J\ \middle|\ \frac1{\sqrt N}C\one\right]\,.
\eqlabel{aug-rect-matrix}
\]

\begin{proposition}
\label{prop:ARI}
With probability $1-o(1)$ under $\cM_{N,d}$,
\[
 \inf_{|J|\leq(1-\alpha)N}s_{\min}(\mathcal B_J(C))>N^{-K_{\rm H}}\,.
\eqlabel{ARI}
\]
The same assertion holds simultaneously with $C$ replaced by
$C^*=A^*-\overline wI$.
\end{proposition}

\begin{proof}
Suppose \eqref{ARI} fails.  Choose $(x,t)$ with $\|(x,t)\|_2=1$ such that,
with
\[
 v=E_Jx+\frac t{\sqrt N}\one\,,
\]
one has $\|Cv\|_2\leq N^{-K_{\rm H}}$.  The Gram matrix of the map
$(x,t)\mapsto v$ has smallest eigenvalue at least
$1-\sqrt{|J|/N}\geq1-\sqrt{1-\alpha}$.  Thus, with
$c_\alpha=(1-\sqrt{1-\alpha})^{1/2}>0$,
\[
 c_\alpha\leq\|v\|_2\leq\sqrt2\,.
\eqlabel{ARI-vector-norm}
\]
Put $z=N^{K_{\rm H}}v$.  Outside $J$ all coordinates of $z$ are equal,
so one level set has size at least $\alpha N$, while
$\|z\|_2\geq c_\alpha N^{K_{\rm H}}$ by
\eqref{ARI-vector-norm}.  Moreover,
\[
 \|Cz\|_\infty\leq N^{K_{\rm H}}\|Cv\|_2\leq1\,.
\]
This contradicts Theorem~\ref{thm:HGALR}.  Applying the theorem to the
transpose model and the conjugate shift gives the assertion for $C^*$;
a union bound makes the two conclusions simultaneous.
\end{proof}

\subsection{Visibility and determinant flow}

Choose fixed exponents
\[
 K_{\rm vis}>K_{\rm H}+\frac12\,,
 \qquad
 K_{\rm PC}>2K_{\rm vis}\,,
 \qquad
 \lambda=N^{-K_{\rm PC}}\,,
 \qquad
 t=N^{-K_{\rm vis}}\,.
\]

\begin{lemma}
\label{lem:spread}
On the simultaneous event in Proposition~\ref{prop:ARI}, every unit
vector $u$ satisfying $\|Cu\|_2\leq\lambda$ has fewer than $\alpha N$
coordinates in every disk of radius $t$.  The same holds for unit
$u$ satisfying $\|C^*u\|_2\leq\lambda$.
\end{lemma}

\begin{proof}
Suppose $u_i\in B(z,t)$ on a set $S$ of size at least $\alpha N$, put
$J=S^c$, and replace $u$ on $S$ by $z$.  The resulting vector has the
form $u'=E_Jx+z\one$, and
\[
 \|u-u'\|_2\leq\sqrt Nt\,,
 \qquad
 \|Cu'\|_2\leq\lambda+(d+|w|)\sqrt Nt\,.
\eqlabel{spread-error}
\]
As $K_{\rm vis}>1/2$, $\|u'\|_2=1-o(1)$.  Its coefficient vector in
\eqref{aug-rect-matrix}, namely
$(x,z\sqrt N)$, consequently has norm bounded below.  Proposition
\ref{prop:ARI} gives $\|Cu'\|_2\geq cN^{-K_{\rm H}}$, contradicting
\eqref{spread-error} because $K_{\rm vis}>K_{\rm H}+1/2$ and
$K_{\rm PC}>K_{\rm H}$.
The adjoint statement is identical.
\end{proof}

Suppose $1\leq s\leq n_\lambda(C)$, and let $V_s$ and $U_s$ be respectively
the bottom $s$-dimensional right and left singular subspaces.  Call an
unordered pair $e=\{i,j\}$ right-visible if
\[
 \|P_{V_s}(e_i-e_j)\|_2\geq t\,,
\]
and left-visible if the analogous inequality holds for $U_s$.

\begin{lemma}
\label{lem:majority}
On the event in Lemma~\ref{lem:spread}, uniformly in $s$ with
$n_\lambda(C)\geq s$, the proportion of pairs which are both left- and
right-visible is at least
\[
 1-2\alpha-o(1)>\frac12\,.
\eqlabel{joint-majority}
\]
\end{lemma}

\begin{proof}
Choose any unit $v\in V_s$.  If $\|P_{V_s}(e_i-e_j)\|_2<t$, then
$|v_i-v_j|<t$.  For each fixed $i$, Lemma~\ref{lem:spread}, applied to
the disk centred at $v_i$, gives fewer than $\alpha N$ such $j$.
Thus at most an $\alpha+o(1)$ proportion of unordered pairs are not
right-visible.  The same argument with one unit vector in $U_s$ gives
the corresponding bound for left visibility.  A union bound proves \eqref{joint-majority}.
\end{proof}

For a pair $e=\{i,j\}$, put $a_e=e_i-e_j$.  Then
$P_e=I-a_ea_e^*$ is the transposition matrix swapping columns $i,j$.
Define
\[
 T_eA=AP_e\,,
 \qquad
 \Psi_\lambda(A)=\log\det(C^*C+\lambda^2I)\,.
\]
The maps $T_e$ are labelled involutions preserving the uniform law on
$\cM_{N,d}$.

\begin{lemma}
\label{lem:determinant}
For every $A$ and $e$,
\[
 \exp(\Psi_\lambda(T_eA)-\Psi_\lambda(A))
 \geq\lambda^2|w|^2L_eR_e\,,
\eqlabel{det-lower}
\]
where
\[
 L_e=a_e^*(CC^*+\lambda^2I)^{-1}a_e\,,
 \qquad
 R_e=a_e^*(C^*C+\lambda^2I)^{-1}a_e\,.
\]
If $n_\lambda(C)\geq s$ and $e$ is jointly visible, then
\[
 \Psi_\lambda(T_eA)>\Psi_\lambda(A)
\eqlabel{det-uphill}
\]
for all sufficiently large $N$.
\end{lemma}

\begin{proof}
Since $P_e$ is unitary,
\[
 (T_eA-wI)P_e=A-wP_e=C+wa_ea_e^*\,.
\eqlabel{transposition-update}
\]
For arbitrary $C,u,v$ and $\lambda>0$, the matrix determinant lemma
applied to the rank-one update gives
\[
 \begin{split}
 &\,\frac{\det((C+uv^*)^*(C+uv^*)+\lambda^2I)}
          {\det(C^*C+\lambda^2I)}\\
          =&\,
 \left|1+v^*C^*(CC^*+\lambda^2I)^{-1}u\right|^2+
 \lambda^2
 (v^*(C^*C+\lambda^2I)^{-1}v)
 (u^*(CC^*+\lambda^2I)^{-1}u)\,.
 \end{split}
\]
Taking $u=wa_e$ and $v=a_e$, then discarding the first nonnegative term,
proves \eqref{det-lower}.

On the bottom right singular subspace, the eigenvalues of
$(C^*C+\lambda^2I)^{-1}$ are at least $(2\lambda^2)^{-1}$; hence right
visibility gives $R_e\geq t^2/(2\lambda^2)$.  Similarly,
$L_e\geq t^2/(2\lambda^2)$.  Therefore
\[
 \exp(\Psi_\lambda(T_eA)-\Psi_\lambda(A))
 \geq\frac{|w|^2t^4}{4\lambda^2}\longrightarrow\infty
\]
by $K_{\rm PC}>2K_{\rm vis}$ and $w\neq0$.  This proves
\eqref{det-uphill}.
\end{proof}

The matrices in \eqref{transposition-update} differ by rank one, up to a unitary factor.
Singular-value interlacing consequently gives
\[
 |n_\lambda(T_eA-wI)-n_\lambda(A-wI)|\leq1\,.
\eqlabel{hardedge-interlacing}
\]
This remains valid at singular matrices and with the strict threshold in
the definition of $n_\lambda$.

\begin{proposition}
\label{prop:boundary}
Under $\cM_{N,d}$, $n_\lambda(A-wI)=O_{\Pp}(1)$.
\end{proposition}

\begin{proof}
Let $G$ be the simultaneous left--right event in
Proposition~\ref{prop:ARI}, and write $\Pp(G^c)=\eta_N=o(1)$.  Choose
\[
 \frac12<\vartheta<1-2\alpha\,.
\]
By Lemmas~\ref{lem:majority} and~\ref{lem:determinant}, every
$A\in G$ with $n_\lambda(C)\geq s$ has at least
$\vartheta D_N$ transposition labels that strictly increase the regularized log determinant, where
$D_N=\binom N2$.

Let
\[
 S_s=\{A:n_\lambda(A-wI)\geq s\}\,,
 \qquad p_s=\Pp(S_s)\,,
 \qquad S=G\cap S_s\,.
\]
Orient every non-loop labelled transposition edge in the direction of
strictly increasing $\Psi_\lambda$; ties are left unoriented.  At each
vertex of $S$, the number of outgoing labels is at least $\vartheta D_N$ and
the number of incoming labels is at most $(1-\vartheta)D_N$.  Summing
outdegree minus indegree over $S$ cancels every internal oriented edge.
Consequently the number of outgoing boundary labels is at least
\[
 (2\vartheta-1)D_N|S|\,.
\eqlabel{boundary-divergence}
\]
By \eqref{hardedge-interlacing}, the target of an outgoing boundary label lies either in
$G^c$ or in $S_{s-1}\setminus S_s$.  Every target state has at most
$D_N$ incoming labels.  Dividing \eqref{boundary-divergence} by the size of the uniform state
space gives
\[
 (2\vartheta-1)(p_s-\eta_N)
 \leq p_{s-1}-p_s+\eta_N\,,
\]
or
\[
 p_s\leq\frac1{2\vartheta}p_{s-1}+\eta_N\,.
\]
Writing $r=(2\vartheta)^{-1}<1$ and iterating from $p_0=1$ gives
\[
 p_s\leq r^s+\eta_N\frac{1-r^s}{1-r}\,.
\]
Taking $N\to\infty$ for fixed $s$ gives
\[
 \limsup_{N\to\infty}p_s\leq(2\vartheta)^{-s}\,.
\]
Letting $s\to\infty$ proves tightness.
\end{proof}

Proposition~\ref{prop:boundary} proves
Theorem~\ref{thm:polynomial-count}, with $K=K_{\rm PC}$, under the uniform law on
$\cM_{N,d}$.  Equation~\eqref{diagonal-positive} and conditioning
transfer the conclusion to $\cM^0_{N,d}$.

\section{Hermitization and the oriented Kesten--McKay law}
\label{sec:girko}

Throughout this section, fix $d\geq2$ and let $A$ be uniform on
$\cM^0_{N,d}$.  For $w\in\C$, define the symmetrized empirical
singular-value measure
\[
 \nu_N^w=\frac1{2N}\sum_{j=1}^N
 \bigl(\delta_{\tau_j(A-wI)}+\delta_{-\tau_j(A-wI)}\bigr)
\]
and the counting function
\[
 F_{N,w}(y)=\frac1N\#\{j:\tau_j(A-wI)<y\}
            =\nu_N^w((-y,y))\,.
\]

\subsection{Hermitized convergence and mesoscopic estimates}

Write
\[
 m_{N,w}(z)=\int_{\R}\frac{1}{s-z}\,\dd\nu_N^w(s)\,.
\]
We first record the local-law estimate used when $d\geq3$.
Apply Theorem~1.3 of Adhikari--Dembo \cite{AdhikariDembo} with its
parameter equal to $1/2$, and let
$\beta_{\rm AD}=\varepsilon(1/2,d)>0$ be the exponent in that theorem.
Lemma~A.1 of the same paper bounds the limiting Stieltjes transform
uniformly.  Hence, for each compact $\mathbb D\subset\C$,
Theorem~1.3 gives events of probability $1-O(N^{-1/2})$ on which
\[
 \sup_{w\in\mathbb D}\sup_{N^{-\beta_{\rm AD}}\leq\eta\leq1}
 \im m_{N,w}(\ii\eta)\leq C_{d,\mathbb D}\,.
\]
For $N^{-\beta_{\rm AD}}\leq y\leq1$, the Poisson-kernel bound gives
\[
 F_{N,w}(y)
 \leq 2y\,\im m_{N,w}(\ii y)
 \leq C_{d,\mathbb D}y\,.
\]
For $0<y<N^{-\beta_{\rm AD}}$ use monotonicity at
$N^{-\beta_{\rm AD}}$, while for $y>1$ the assertion is trivial.
Thus, on the same events,
\[
 F_{N,w}(y)\leq C_{d,\mathbb D}
 (y\vee N^{-\beta_{\rm AD}})
\eqlabel{AD-count}
\]
simultaneously for $w\in\mathbb D$ and $y>0$.

In degree two, we use a calculation for free Haar unitaries and
quantitative weak convergence to prove the required averaged estimate.

\begin{lemma}
\label{lem:d2-free-poisson}
Let $(\mathcal M,\varphi)$ be a tracial $W^*$-probability space containing
free Haar unitaries $u_1,u_2$, and let $\operatorname{tr}_2$ be the
normalized trace on $2\times2$ matrices.  Put
\[
 H_w=\begin{pmatrix}0&u_1+u_2-w\\
 u_1^*+u_2^*-\overline w&0\end{pmatrix}\,,
 \qquad
 M_w(\eta)=(\operatorname{tr}_2\otimes\varphi)
       \frac{\eta}{H_w^2+\eta^2}
 =\int_{\R}\frac{\eta}{s^2+\eta^2}\,\dd\nu_2^w(s)\,,
\]
where $\nu_2^w$ is the spectral distribution of $H_w$ with respect to
$\operatorname{tr}_2\otimes\varphi$.
Then
\[
 \sup_{w\in\C}\sup_{0<\eta\leq1}M_w(\eta)<\infty\,.
 \eqlabel{d2-free-poisson-bound}
\]
More precisely, if $r=|w|^2$, then
\[
 M_w(\eta)=\frac{\eta+2q_\eta}
 {r+(\eta+2q_\eta)^2}\,,
 \eqlabel{d2-free-poisson-formula}
\]
where $q_\eta$ is the unique positive solution of
\[
 q_\eta((r-2)+\eta^2+3\eta q_\eta+2q_\eta^2)=\eta\,.
 \eqlabel{d2-free-cubic}
\]
\end{lemma}

\begin{proof}
After a diagonal gauge, $H_w$ is the weighted adjacency operator of
the bipartite $3$-regular tree whose incident edge weights are
$1,1,|w|$; the underlying tree is the universal cover of the
two-vertex multigraph with three parallel edges.  Let $p_\eta$ and
$q_\eta$ be $-\ii$ times the diagonal
resolvents at $\ii\eta$ after deleting, respectively, the
$|w|$-edge and a unit edge.  Schur complementation gives
\[
 p_\eta=\frac1{\eta+2q_\eta}\,,
 \qquad
 q_\eta=\frac1{\eta+q_\eta+rp_\eta}\,.
\]
Eliminating $p_\eta$ gives \eqref{d2-free-cubic}, and the diagonal
resolvent at the root gives \eqref{d2-free-poisson-formula}.
The polynomial obtained by subtracting $\eta$ from the left-hand side
of \eqref{d2-free-cubic} is negative at zero and tends to infinity.
Its derivative is strictly increasing on $(0,\infty)$, so it has
exactly one positive root.
Put $t=\eta+2q_\eta$.  If $r\geq1$, then
$t/(r+t^2)\leq(2\sqrt r)^{-1}$.  If $0\leq r\leq1$, positivity in
\eqref{d2-free-cubic} implies
$\eta^2+3\eta q_\eta+2q_\eta^2>2-r\geq1$, and hence $t\geq1$ and
$t/(r+t^2)\leq1$.  This proves \eqref{d2-free-poisson-bound}.
\end{proof}

Let $\mathcal S_N$ be the permutations of $[N]$.  For $\theta>0$,
write $\operatorname{Ewens}_N(\theta)$ for the law
\[
 \Pp_\theta(\sigma)
 =\frac{\theta^{c(\sigma)}}{\theta(\theta+1)\cdots(\theta+N-1)}\,,
 \qquad \sigma\in\mathcal S_N\,,
\]
where $c(\sigma)$ is the number of cycles.

\begin{lemma}
\label{lem:d2-ewens-rank}
Let $P$ and $Q$ be the permutation matrices of two independent
permutations of $[N]$, the first uniform and the second with law
$\operatorname{Ewens}_N(1/2)$.  Conditional on the second permutation
having no fixed point,
\[
 A=P(I+Q)
 \eqlabel{d2-ewens-representation}
\]
is uniform on $\cM_{N,2}$.  Moreover, under its unconditioned Ewens
law, $Q$ can be coupled to a uniform permutation matrix $U$ so that
\[
 \E\rank(Q-U)=O(\log N)\,.
 \eqlabel{d2-ewens-rank-bound}
\]
\end{lemma}

\begin{proof}
If $Q$ has no fixed point, the two permutation matrices $P$ and $PQ$
have disjoint supports.  Conversely, if $A\in\cM_{N,2}$ has $c(A)$
components as a two-regular bipartite graph, it has exactly
$2^{c(A)}$ ordered decompositions into two perfect matchings.  In every
decomposition $A=P+PQ$, the relative permutation $Q$ has $c(A)$
cycles.  The Ewens weight of each decomposition is proportional to
$2^{-c(A)}$, so summing over its $2^{c(A)}$ decompositions gives the
same total weight to every $A$.  This proves
\eqref{d2-ewens-representation}.

For the coupling, use the Chinese-restaurant construction
simultaneously for parameters $1/2$ and $1$.  At stage $n$, the
probabilities of starting a new cycle are respectively
\[
 \frac1{2n-1}\quad\hbox{and}\quad\frac1n\,.
\]
Couple these decisions monotonically and, if both constructions insert
$n$ into an old cycle, use the same uniformly chosen predecessor.
Insertion after $j$ is right multiplication by the transposition
$(j\ n)$.  Hence the Cayley distance increases by at most one at a
disagreement and is unchanged otherwise.  The expected number of
disagreements is at most $\sum_{n\leq N}n^{-1}=O(\log N)$, and a
single transposition changes a permutation matrix by rank one.  This
proves \eqref{d2-ewens-rank-bound}.
\end{proof}

\begin{proposition}
\label{prop:d2-mesoscopic-green}
Let $A$ be uniform on either $\cM_{N,2}$ or $\cM^0_{N,2}$.  For every
compact $\mathbb D\subset\C$, uniformly for $w\in\mathbb D$ and
$0<\eta\leq1$,
\[
 \E\,\im m_{N,w}(\ii\eta)
 \leq C_{\mathbb D}+\frac{C_{\mathbb D}}{N\eta^6}
       +\frac{C_{\mathbb D}\log N}{N\eta}\,.
 \eqlabel{d2-green-estimate}
\]
In particular, for $\eta_0=N^{-1/7}$,
\[
 \sup_{w\in\mathbb D}\E\,\im m_{N,w}(\ii\eta_0)
 =O_{\mathbb D}(1)\,.
 \eqlabel{d2-green-one-point}
\]
For every fixed $w$, $\nu_N^w$ also converges weakly in probability to
the free Hermitized law $\nu_2^w$.
\end{proposition}

\begin{proof}
For an element $B$ of a unital $*$-algebra, put
\[
 H(B)=\begin{pmatrix}0&B\\B^*&0\end{pmatrix}\,,
 \qquad h_\eta(x)=\frac{\eta}{x^2+\eta^2}\,.
\]
For unitaries $x,y$, write $X_w(x,y)=H(1+x-wy)$.
The definition of $m_{N,w}$ gives
\[
 \im m_{N,w}(\ii\eta)
 =\frac1{2N}\operatorname{Tr}h_\eta(H(A-wI))\,.
\]
Let $U,R$ be independent uniform permutation matrices and put
\[
 X_w(U,R)=H(I+U-wR)=
 \begin{pmatrix}
 0&I+U-wR\\ I+U^*-\overline wR^*&0
 \end{pmatrix}\,.
 \eqlabel{d2-hermitian-polynomial}
\]
Let $(\mathcal M,\varphi)$ be a tracial $W^*$-probability space containing
free Haar unitaries $v_1,v_2$, corresponding to $U,R$.
We use the polynomial master inequality of
Chen--Garza-Vargas--Tropp--van Handel \cite[Theorem~6.1]{CGVTH}
and the Chebyshev argument proving \cite[Corollary~7.2]{CGVTH}.
The full group
$C^*$-norm required there, with $u_1,u_2$ denoting the canonical
generators of the full group $C^*$-algebra of the free group
$\mathbf F_2$, is
\[
 K_w=\|X_w(u_1,u_2)\|_{M_2(\C)\otimes C^*(\mathbf F_2)}=2+|w|:
\]
the triangle inequality gives the upper bound, while scalar unitaries
attain it.  For $h\in C^\infty(\R)$, set
\[
 f_w(\theta)=h(K_w\cos\theta)\,.
\]
Taking $\beta=2$, we obtain
\[
 \left|
 \E\frac1{2N}\operatorname{Tr} h
   (X_w(U,R)|_{\one^\perp\oplus\one^\perp})
 -(\operatorname{tr}_2\otimes\varphi)h(X_w(v_1,v_2))
 \right|
 \leq\frac{C_{\mathbb D}}{N}
       \bigl(\|f_w^{(5)}\|_{L^2[0,2\pi]}
             +\|h\|_{C^0[-K_w,K_w]}\bigr)\,.
 \eqlabel{d2-CGVTH}
\]
The normalization in \cite{CGVTH} is by $2N$, also on the
$2(N-1)$-dimensional restricted space.  To obtain
\eqref{d2-CGVTH} with this normalization, expand
$h(x)=\sum_{j\geq0}a_jT_j(x/K_w)$ in Chebyshev polynomials.
The constant term contributes exactly $-a_0/N$ to the difference
of traces.  For $j\geq1$, \cite[Theorem~6.1]{CGVTH}, with $m=1$,
bounds the corresponding difference by $C_{\mathbb D}j^4/N$.
Moreover, \cite[Lemma~4.4]{CGVTH} gives
\[
 |a_0|\leq\|h\|_{C^0[-K_w,K_w]}\,,
 \qquad
 \sum_{j\geq1}j^4|a_j|
 \leq C\|f_w^{(5)}\|_{L^2[0,2\pi]}\,.
\]
Summing these bounds proves \eqref{d2-CGVTH}.

Now take $h=h_\eta$ and write
$f_{\eta,w}(\theta)=h_\eta(K_w\cos\theta)$.  Since
$K_w\leq2+\sup_{z\in\mathbb D}|z|$ and
$\|h_\eta^{(j)}\|_\infty\leq C_j\eta^{-j-1}$ for $0\leq j\leq5$,
the chain rule gives
\[
 \|f_{\eta,w}^{(5)}\|_{L^2[0,2\pi]}
 \leq C_{\mathbb D}\eta^{-6}\,.
\]
Also $\|h_\eta\|_{C^0[-K_w,K_w]}=\eta^{-1}$, so the
normalization correction in \eqref{d2-CGVTH} is at most
$C_{\mathbb D}/(N\eta)$.
On the constant-vector block, \eqref{d2-hermitian-polynomial} has
eigenvalues $\pm|2-w|$, so its contribution to the full
normalized trace is
\[
 \frac{\eta}{N(|2-w|^2+\eta^2)}\leq\frac1{N\eta}\,.
\]
Since
\[
 (1+v_1-wv_2)v_2^*=v_2^*+v_1v_2^*-w
\]
and $(v_2^*,v_1v_2^*)$ is a free Haar pair, right multiplication by
$v_2^*$ identifies its Hermitized law with $\nu_2^w$.
Lemma~\ref{lem:d2-free-poisson} therefore shows
that the limiting term is bounded.  We have proved
\[
 \E\frac1{2N}\operatorname{Tr} h_\eta(X_w(U,R))
 \leq C_{\mathbb D}+\frac{C_{\mathbb D}}{N\eta^6}
       +\frac{C_{\mathbb D}}{N\eta}\,.
\]

Now let $P$ be uniform and let
$Q\sim\operatorname{Ewens}_N(1/2)$ have its unconditioned law,
independently of $P$, and put $\widetilde A=P(I+Q)$.  Since left
multiplication preserves singular values and $R=P^*$ is uniform and
independent of $Q$,
\[
 \tau_j(\widetilde A-wI)=\tau_j(I+Q-wR)\,,
 \qquad 1\leq j\leq N\,.
\]
Couple $Q$ to a uniform $U$, independently of $R$, and put
\[
 H_Q=X_w(Q,R)\,,\qquad H_U=X_w(U,R)\,.
\]
The two matrices differ by rank at most $2\rank(Q-U)$.  The rank
inequality for Hermitian empirical distributions gives
\[
 \left|\frac1{2N}\operatorname{Tr} h_\eta(H_Q)
       -\frac1{2N}\operatorname{Tr} h_\eta(H_U)\right|
 \leq\frac{2\rank(Q-U)}{N\eta}\,.
\]
Taking expectations and using \eqref{d2-ewens-rank-bound} proves
\eqref{d2-green-estimate} for the unconditioned matrix
$\widetilde A$.  Under the Ewens-$1/2$ law the probability of no fixed
point tends to $e^{-1/2}>0$.  Since $h_\eta\geq0$,
conditioning changes the upper bound by at most a fixed factor, and
Lemma~\ref{lem:d2-ewens-rank} identifies the resulting conditional
matrix with a uniform member of $\cM_{N,2}$.  A further conditioning,
using \eqref{diagonal-positive}, proves the estimate for
$\cM^0_{N,2}$.

For weak convergence, take a fixed $h\in C_c^\infty(\R)$ in
\eqref{d2-CGVTH}.  The comparison error is $O_{\mathbb D,h}(N^{-1})$;
adding the constant-vector block contributes at most
$\|h\|_\infty/N$ and gives the same order of error for the full trace.
Composing either permutation by a
transposition changes the Hermitization by rank at most two, and hence
changes the normalized trace by at most
$2\operatorname{TV}(h)/N$, where $\operatorname{TV}(h)$ denotes total
variation.  The permutation-exposure martingale and
Azuma's inequality therefore give concentration about the expectation.
The same rank inequality and \eqref{d2-ewens-rank-bound} show that the
Ewens coupling changes a normalized smooth trace by $o_{\Pp}(1)$, and
both conditioning events have probabilities bounded away from zero.
Uniform approximation on the deterministic spectral interval
$[-2-|w|,2+|w|]$ gives weak convergence for every continuous test
function.  The limit is $\nu_2^w$ by the change of free generators above.
\end{proof}

Thus, for every $d\geq2$ and every fixed $w$, $\nu_N^w$ converges
weakly in probability to a deterministic symmetric measure $\nu_d^w$:
this is Theorem~1.3 of \cite{AdhikariDembo} when $d\geq3$ and
Proposition~\ref{prop:d2-mesoscopic-green} when $d=2$.

Since $\|A-wI\|\leq d+|w|$, weak convergence applies to every fixed
truncation of the logarithm.  Thus, for fixed $T>0$,
\[
 \frac1N\sum_{j=1}^N\max(\log\tau_j(A-wI),-T)
 \longrightarrow
 V_{d,T}(w):=\int\max(\log|s|,-T)\,\dd\nu_d^w(s)
 \eqlabel{trunc-conv}
\]
in probability.  Put
\[
 U_d(w)=\int\log|s|\,\dd\nu_d^w(s)\,.
\]
For $d\geq3$, we apply Lemma~A.1 of \cite{AdhikariDembo} and the
Poisson-kernel bound.  For $d=2$, we apply
Lemma~\ref{lem:d2-free-poisson}.  In both cases,
for every compact $\mathbb D\subset\C$, uniformly for
$w\in\mathbb D$ and $0<y<1$,
\[
 \nu_d^w((-y,y))\leq C_{d,\mathbb D}y\,.
\]
Consequently,
\[
 \begin{split}
 0\leq V_{d,T}(w)-U_d(w)
 =\int_T^\infty\nu_d^w((-e^{-u},e^{-u}))\,\dd u\leq C_{d,\mathbb D}e^{-T}\,.
 \end{split}
 \eqlabel{limit-untruncation}
\]
The limiting potential, calculated for the free sum of $d$ Haar
unitaries in \cite{BasakDembo} and identified with the tree limit in
\cite{AdhikariDembo}, is radial.  With $r=|w|$,
\[
 U_d(r)=
 \begin{cases}
 \displaystyle
 \frac12\log d-\frac{d-1}{2}
 \log\left(\frac{d^2-r^2}{d^2-d}\right)\,,
      &0\leq r\leq\sqrt d\,,\\[2mm]
 \log r\,,&r\geq\sqrt d\,.
 \end{cases}
 \eqlabel{limit-potential}
\]

\subsection{Logarithmic uniform integrability}

For $T>0$, define
\[
 R_{N,T}(w)=\frac1N\sum_{j=1}^N
 \left(\log_+\frac1{\tau_j(A-wI)}-T\right)_+\,,
\]
with value $+\infty$ if $A-wI$ is singular.

\begin{proposition}
\label{prop:UI}
Let $\mathbb D\subset\C$ be compact and let
$\varnothing\neq\mathcal W\subset
\mathbb D\cap(\Q(i)\setminus\Z[i])$ be finite.  For every fixed
$\delta>0$,
\[
 \lim_{T\to\infty}\limsup_{N\to\infty}
 \Pp\left\{\max_{w\in\mathcal W}R_{N,T}(w)>\delta\right\}=0\,.
 \eqlabel{UI}
\]
More precisely, if $(b_w)_{w\in\mathcal W}$ are nonnegative fixed
weights, then, for every fixed $T,\delta>0$,
\[
 \limsup_{N\to\infty}
 \Pp\left\{\sum_{w\in\mathcal W}b_wR_{N,T}(w)>\delta\right\}
 \leq\frac{C_{d,\mathbb D}e^{-T}}{\delta}
       \sum_{w\in\mathcal W}b_w\,.
 \eqlabel{weighted-UI}
\]
\end{proposition}

\begin{proof}
Every $A-wI$ is nonsingular for $w\in\mathcal W$.  Indeed, an
eigenvalue of the integer matrix $A$ is an algebraic integer, whereas
the algebraic integers in $\Q(i)$ are precisely $\Z[i]$.

Choose $K>1/7$ larger than the finitely many exponents in
Theorem~\ref{thm:polynomial-count}, one for each $w\in\mathcal W$.
Then
\[
 \max_{w\in\mathcal W}n_{N^{-K}}(A-wI)=O_{\Pp}(1)\,.
 \eqlabel{grid-polynomial-count}
\]

Put
\[
 X_{N,w}=\frac1N\log_+\frac1{s_{\min}(A-wI)}\,.
\]
Since $\mathcal W\subset\Q(i)\setminus\Z[i]$, every point satisfies the
hypotheses of Theorem~\ref{thm:small}.  A union bound therefore gives
\[
 \max_{w\in\mathcal W}X_{N,w}=o_{\Pp}(1)\,.
 \eqlabel{grid-endpoint-depth}
\]

Set $h_N=K\log N$.  The layer-cake identity gives
\[
 R_{N,T}(w)=\int_T^\infty F_{N,w}(e^{-u})\,\dd u\,.
 \eqlabel{layer-cake}
\]
Suppose first that $d\geq3$.  On the event \eqref{AD-count}, uniformly
over the finite set $\mathcal W$,
\[
 \int_T^{h_N}F_{N,w}(e^{-u})\,\dd u
 \leq C_{d,\mathbb D}e^{-T}
       +C_{d,\mathbb D}K(\log N)N^{-\beta_{\rm AD}}
 =C_{d,\mathbb D}e^{-T}+o(1)\,.
 \eqlabel{polynomial-layer}
\]
For $x_j(w)=\log_+\tau_j(A-wI)^{-1}$, the remaining integral satisfies
\[
 \begin{split}
 \int_{h_N}^\infty F_{N,w}(e^{-u})\,\dd u
 =\frac1N\sum_{j=1}^N(x_j(w)-h_N)_+\leq n_{N^{-K}}(A-wI)X_{N,w}\,.
 \end{split}
 \eqlabel{terminal-layer}
\]
Equations \eqref{grid-polynomial-count} and
\eqref{grid-endpoint-depth} show that the last expression is $o_{\Pp}(1)$,
uniformly over the finite set.  Combining
\eqref{polynomial-layer} and \eqref{terminal-layer} proves
\eqref{weighted-UI}, after increasing $C_{d,\mathbb D}$ if necessary,
for $d\geq3$.

It remains to treat $d=2$.  Put
\[
 \eta_0=N^{-1/7}\,,\qquad u_0=\frac17\log N\,,
\]
and, for $N$ large enough that $T<u_0<h_N$, split
\eqref{layer-cake} over
\[
 [T,u_0]\,,\qquad[u_0,h_N]\,,\qquad[h_N,\infty)\,,
\]
denoting the three terms by $I_1(w),I_2(w),I_3(w)$.  The Poisson-kernel
bound is
\[
 F_{N,w}(y)\leq2y\,\im m_{N,w}(\ii y)\,.
 \eqlabel{d2-poisson-count}
\]
The constants in Proposition~\ref{prop:d2-mesoscopic-green} are
uniform for $w\in\mathbb D$ by \eqref{d2-free-cubic} and the
uniformity of \eqref{d2-CGVTH} on compact sets.  That proposition,
Tonelli's theorem, and the change
of variables $y=e^{-u}$ give
\[
 \begin{split}
 \E I_1(w)
 &\leq2\int_{\eta_0}^{e^{-T}}
       \E\,\im m_{N,w}(\ii y)\,\dd y\\
 &\leq C_{\mathbb D}e^{-T}
       +\frac{C_{\mathbb D}}{N\eta_0^5}
       +\frac{C_{\mathbb D}(\log N)^2}{N}
 =C_{\mathbb D}e^{-T}+o(1)\,.
 \end{split}
 \eqlabel{d2-upper-window}
\]
For $0<y\leq\eta_0$, comparison of the Poisson kernels gives
\[
 \im m_{N,w}(\ii y)
 \leq\frac{\eta_0}{y}\im m_{N,w}(\ii\eta_0)\,.
\]
Thus \eqref{d2-poisson-count} and
\eqref{d2-green-one-point} imply
\[
 \begin{split}
 \max_{w\in\mathcal W}I_2(w)
 &\leq2K(\log N)\eta_0
       \max_{w\in\mathcal W}\im m_{N,w}(\ii\eta_0)\\
 &=o_{\Pp}(1)\,.
 \end{split}
\]
Finally, \eqref{grid-polynomial-count} and
\eqref{grid-endpoint-depth} give
$\max_{w\in\mathcal W}I_3(w)=o_{\Pp}(1)$.  Multiply
\eqref{d2-upper-window} by $b_w$, sum over $w$, and apply Markov's
inequality.  The second and third terms are $o_{\Pp}(1)$ for the
fixed finite set, so
\[
 \limsup_{N\to\infty}
 \Pp\left\{\sum_{w\in\mathcal W}b_wR_{N,T}(w)>\delta\right\}
 \leq\frac{C_{2,\mathbb D}e^{-T}}{\delta}
       \sum_{w\in\mathcal W}b_w\,.
\]
This is \eqref{weighted-UI}.  Taking $b_w=1$ for every $w$ and then
letting $T\to\infty$ proves \eqref{UI}.
\end{proof}

\subsection{The finite-grid argument}

For an $N\times N$ matrix $B$ with eigenvalues
$\zeta_1,\ldots,\zeta_N$, counted with algebraic multiplicity, put
\[
 L_B=\frac1N\sum_{j=1}^N\delta_{\zeta_j}\,,
 \qquad
 U_B(w)=\frac1N\log|\det(B-wI)|
\]
and
\[
 V_{B,T}(w)=\frac1N\sum_{j=1}^N
       \max(\log\tau_j(B-wI),-T)\,.
\]

\begin{lemma}\label{lem:grid}
Let $T,h>0$, let $f\in C_c^2(\C)$ be real-valued, and put
$g=(2\pi)^{-1}\Delta f$.  Let $C_\alpha$ be the finitely many
half-open squares of a square grid which meet $\operatorname{supp}g$,
each of diameter at most $h$.  Choose $w_\alpha\in C_\alpha$ and set
\[
 a_\alpha=\int_{C_\alpha}g(w)\,\dd^2w\,.
 \eqlabel{grid-weights}
\]
If $B-w_\alpha I$ is nonsingular for every $\alpha$, then
\begin{equation}  \label{grid-bound}
 \Big|\int f\,\dd L_B-
       \sum_\alpha a_\alpha V_{B,T}(w_\alpha)\Big|
 \leq\frac14\|\Delta f\|_\infty e^{-2T}
       +\frac{he^T}{2\pi}\|\Delta f\|_1+\sum_\alpha|a_\alpha|
       \bigl(V_{B,T}(w_\alpha)-U_B(w_\alpha)\bigr)\,.
\end{equation}

\end{lemma}

\begin{proof}
Put
\[
 \ell_T(z)=\max(\log|z|,-T)\,,
 \qquad
 \widetilde V_{B,T}(w)=\frac1N\sum_{j=1}^N\ell_T(\zeta_j-w)\,.
\]
Girko's distributional identity \cite{Girko} gives
\[
 \begin{gathered}
 \int f\,\dd L_B-
 \sum_\alpha a_\alpha V_{B,T}(w_\alpha)
 =\int g(w)(U_B(w)-\widetilde V_{B,T}(w))\,\dd^2w\\
 {}+\sum_\alpha\int_{C_\alpha}g(w)
 (\widetilde V_{B,T}(w)-\widetilde V_{B,T}(w_\alpha))\,\dd^2w+\sum_\alpha a_\alpha
 (\widetilde V_{B,T}(w_\alpha)-V_{B,T}(w_\alpha))\,.
 \end{gathered}
\]
The first term has absolute value at most
\[
 \|g\|_\infty
 \int_{|z|<e^{-T}}(-\log|z|-T)\,\dd^2z
 =\frac14\|\Delta f\|_\infty e^{-2T}\,.
\]
Since $\ell_T$ is $e^T$-Lipschitz, the second term has absolute value
at most
\[
 he^T\|g\|_1=\frac{he^T}{2\pi}\|\Delta f\|_1\,.
\]
Finally, multiplicative Weyl majorization and the convex function
$x\mapsto\max(x,-T)$ give
\[
 0\leq V_{B,T}(w_\alpha)-\widetilde V_{B,T}(w_\alpha)
 \leq V_{B,T}(w_\alpha)-U_B(w_\alpha)\,.
\]
These three estimates prove \eqref{grid-bound}.
\end{proof}

\begin{proof}[Proof of Theorem~\ref{thm:okm}]
Fix a real-valued $f\in C_c^2(\C)$ and put
$g=(2\pi)^{-1}\Delta f$.  If $g=0$, then $f=0$ and there is nothing
to prove.  Choose a compact set $\mathbb D$ containing
the closed unit neighborhood of $\operatorname{supp}g$.  Fix $T>0$
and $0<h\leq1$, and form the grid in Lemma~\ref{lem:grid}.  In each
of these cells choose
\[
 w_\alpha\in\Q(i)\setminus\Z[i]\,.
\]
Such a point exists because $\Q(i)\setminus\Z[i]$ is dense.  Every
chosen point belongs to $\mathbb D$.  Moreover, $A-w_\alpha I$ is nonsingular for every
integer matrix $A$, as shown in the proof
of Proposition~\ref{prop:UI}.

Let $\mathcal W=\{w_\alpha\}$.  Its elements are distinct because the
cells are half-open.  Since $V_{A,T}(w)-U_A(w)=R_{N,T}(w)$ and
$\sum_\alpha|a_\alpha|\leq(2\pi)^{-1}\|\Delta f\|_1$,
\eqref{weighted-UI} with $b_{w_\alpha}=|a_\alpha|$ gives
\[
 \limsup_{N\to\infty}\Pp\left\{
 \sum_\alpha|a_\alpha|
 \bigl(V_{A,T}(w_\alpha)-U_A(w_\alpha)\bigr)>\delta\right\}
 \leq \frac{C_{d,\mathbb D}\|\Delta f\|_1e^{-T}}
 {2\pi\delta}
 \eqlabel{grid-tail}
\]
for every fixed $\delta>0$.
For fixed $T,h$ and the corresponding finite grid,
\eqref{trunc-conv} gives
\[
 \sum_\alpha a_\alpha V_{A,T}(w_\alpha)
 \longrightarrow
 \sum_\alpha a_\alpha V_{d,T}(w_\alpha)
 \eqlabel{grid-limit}
\]
in probability.

Let $\omega_{U_d,\mathbb D}$ be the modulus of continuity of $U_d$ on
$\mathbb D$.  Equations \eqref{limit-untruncation} and
\eqref{grid-weights} yield
\[
 \begin{split}
 \left|\sum_\alpha a_\alpha V_{d,T}(w_\alpha)
       -\int g(w)U_d(w)\,\dd^2w\right|
 &\leq\frac{C_{d,\mathbb D}}{2\pi}
       \|\Delta f\|_1e^{-T}+\omega_{U_d,\mathbb D}(h)\|g\|_1\,.
 \end{split}
 \eqlabel{limit-grid-error}
\]

The two formulas in \eqref{limit-potential} have the same value and
the same radial derivative at $r=\sqrt d$.  Thus $U_d$ is $C^1$ across
the circle, and direct differentiation gives
\[
 \frac1{2\pi}\Delta U_d(z)
 =\frac{d^2(d-1)}{\pi(d^2-|z|^2)^2}
   \mathbf1_{\{|z|<\sqrt d\}}=h_d(z)\,.
\]
In particular, there is no boundary mass on $|z|=\sqrt d$, and
\[
 \int g(w)U_d(w)\,\dd^2w=\int f\,\dd\mu_d\,.
 \eqlabel{limit-id}
\]

Put
\[
 \mathcal R_{N,T,h}
 =\sum_\alpha|a_\alpha|
 \bigl(V_{A,T}(w_\alpha)-U_A(w_\alpha)\bigr)\,.
\]
Combining \eqref{grid-bound}, \eqref{grid-limit},
\eqref{limit-grid-error}, and \eqref{limit-id} shows that
{\medmuskip=1.5mu \thickmuskip=3mu
\[
 \begin{split}
 \left|\int f\,\dd L_A-\int f\,\dd\mu_d\right|
 \leq\tfrac14\|\Delta f\|_\infty e^{-2T}
       +\tfrac{he^T}{2\pi}\|\Delta f\|_1+\tfrac{C_{d,\mathbb D}}{2\pi}
       \|\Delta f\|_1e^{-T}
       +\omega_{U_d,\mathbb D}(h)\|g\|_1
       +\mathcal R_{N,T,h}
       +o_{\Pp}(1)\,.
 \end{split}
\]
}
Given $\varepsilon>0$, first choose $T$ large enough that the terms
containing $e^{-T}$ are small and, by \eqref{grid-tail}, the limiting
probability that $\mathcal R_{N,T,h}>\varepsilon$ is small, uniformly
in the subsequently chosen grid.  Next choose $h$ so that $he^T$ and
$\omega_{U_d,\mathbb D}(h)$ are sufficiently small.  Keep the
resulting finite grid fixed and let $N\to\infty$.  We obtain
\[
 \int f\,\dd L_A\longrightarrow\int f\,\dd\mu_d
\]
in probability.  Since every eigenvalue of $A$ lies in
$\{|z|\leq d\}$, the empirical measures are deterministically tight.
Choose a countable dense convergence-determining subclass of
$C_c^2(\C)$.  The preceding convergence for each member, together with
deterministic tightness, is equivalent to weak convergence in
probability, and the theorem follows.
\end{proof}


{\small
\begin{thebibliography}{99}
\raggedright

\bibitem{AdhikariDembo}
A.~Adhikari and A.~Dembo,
\emph{Spectral measure for uniform $d$-regular digraphs},
arXiv:2310.14132, 2025.

\bibitem{Bai97}
Z.~D.~Bai, \emph{Circular law}, Ann. Probab. \textbf{25} (1997), 494--529.

\bibitem{BR19}
A.~Basak and M.~Rudelson, \emph{The circular law for sparse non-Hermitian matrices}, Ann. Probab. \textbf{47} (2019),
2359--2416.

\bibitem{BasakCookZeitouni}
A.~Basak, N.~Cook, and O.~Zeitouni,
\emph{Circular law for the sum of random permutation matrices},
Electron. J. Probab. \textbf{23} (2018), paper no.~33, 1--51.

\bibitem{BasakDembo}
A.~Basak and A.~Dembo,
\emph{Limiting spectral distribution of sums of unitary and orthogonal
matrices},
Electron. Commun. Probab. \textbf{18} (2013), no.~69, 1--19.

\bibitem{BordenaveChafai}
C.~Bordenave and D.~Chafa\"{\i},
\emph{Around the circular law},
Probab. Surv. \textbf{9} (2012), 1--89.

\bibitem{BordenaveCollins}
C.~Bordenave and B.~Collins,
\emph{Eigenvalues of random lifts and polynomials of random permutation
matrices},
Ann. of Math. (2) \textbf{190} (2019), no.~3, 811--875.

\bibitem{CGVTH}
C.-F.~Chen, J.~Garza-Vargas, J.~A.~Tropp, and R.~van Handel,
\emph{A new approach to strong convergence},
Ann. of Math. (2) \textbf{203} (2026), 555--602.

\bibitem{CookCircular}
N.~A. Cook,
\emph{The circular law for random regular digraphs},
Ann. Inst. H. Poincar\'e Probab. Statist. \textbf{55} (2019), no.~4,
2111--2167.

\bibitem{CookSingularity}
N.~A. Cook,
\emph{On the singularity of adjacency matrices for random regular
digraphs},
Probab. Theory Related Fields \textbf{167} (2017), no.~1--2, 143--200.

\bibitem{Coste}
S.~Coste,
\emph{The spectral gap of sparse random digraphs},
Ann. Inst. H. Poincar\'e Probab. Statist. \textbf{57} (2021), no.~2,
644--684.

\bibitem{CosteLambertZhu}
S.~Coste, G.~Lambert, and Y.~Zhu,
\emph{The characteristic polynomial of sums of random permutations and
regular digraphs},
Int. Math. Res. Not. IMRN (2024), no.~3, 2461--2510.

\bibitem{ErdosSimonovits}
P.~Erd\H{o}s and M.~Simonovits,
\emph{Supersaturated graphs and hypergraphs},
Combinatorica \textbf{3} (1983), no.~2, 181--192.

\bibitem{Girko}
V.~L.~Girko,
\emph{Circular law},
Theory Probab. Appl. \textbf{29} (1985), no.~4, 694--706.

\bibitem{GT10}
F.~G\"{o}tze and A.~Tikhomirov, \emph{The circular law for random matrices}, Ann. Probab. \textbf{38} (2010), 1444--1491.

\bibitem{HaagerupLarsen}
U.~Haagerup and F.~Larsen,
\emph{Brown's spectral distribution measure for $R$-diagonal elements in
finite von Neumann algebras},
J. Funct. Anal. \textbf{176} (2000), no.~2, 331--367.

\bibitem{HuangInvertibility}
J.~Huang,
\emph{Invertibility of adjacency matrices for random $d$-regular graphs},
Duke Math. J. \textbf{170} (2021), no.~18, 3977--4032.

\bibitem{Kesten}
H.~Kesten,
\emph{Symmetric random walks on groups},
Trans. Amer. Math. Soc. \textbf{92} (1959), no.~2, 336--354.

\bibitem{LLTTYCircular}
A.~E. Litvak, A.~Lytova, K.~Tikhomirov, N.~Tomczak-Jaegermann, and
P.~Youssef,
\emph{Circular law for sparse random regular digraphs},
J. Eur. Math. Soc. (JEMS) \textbf{23} (2021), no.~2, 467--501.

\bibitem{LLTTYShifted}
A.~E. Litvak, A.~Lytova, K.~Tikhomirov, N.~Tomczak-Jaegermann, and
P.~Youssef,
\emph{The smallest singular value of a shifted $d$-regular random square
matrix},
Probab. Theory Related Fields \textbf{173} (2019), no.~3--4, 1301--1347.

\bibitem{LLTTYStructure}
A.~E. Litvak, A.~Lytova, K.~Tikhomirov, N.~Tomczak-Jaegermann, and
P.~Youssef,
\emph{Structure of eigenvectors of random regular digraphs},
Trans. Amer. Math. Soc. \textbf{371} (2019), no.~11, 8097--8172.

\bibitem{McKay}
B.~D. McKay,
\emph{The expected eigenvalue distribution of a large regular graph},
Linear Algebra Appl. \textbf{40} (1981), 203--216.

\bibitem{Meszaros}
A.~M\'esz\'aros,
\emph{The distribution of sandpile groups of random regular graphs},
Trans. Amer. Math. Soc. \textbf{373} (2020), 6529--6594.

\bibitem{MRRW}
M.~S. O. Molloy, H.~Robalewska, R.~W. Robinson, and N.~C. Wormald,
\emph{1-factorizations of random regular graphs},
Random Structures Algorithms \textbf{10} (1997), 305--321.

\bibitem{NguyenPan}
H.~H. Nguyen and A.~Pan,
\emph{A note on the singularity probability of random directed
$d$-regular graphs},
European J. Combin. \textbf{122} (2024), paper no.~104039.

\bibitem{RT19}
M.~Rudelson and K.~Tikhomirov, \emph{The sparse circular law under minimal assumptions}, Geom. Funct. Anal. \textbf{29} (2019), 561--637.

\bibitem{SahSahasrabudheSawhney}
A.~Sah, J.~Sahasrabudhe, and M.~Sawhney,
\emph{The limiting spectral law for sparse iid matrices},
Forum Math. Pi, to appear, arXiv:2310.17635v2, 2025.

\bibitem{TaoVu}
T.~Tao and V.~Vu, with an appendix by M.~Krishnapur,
\emph{Random matrices: universality of ESDs and the circular law},
Ann. Probab. \textbf{38} (2010), no.~5, 2023--2065.


\bibitem{Wood}
P.~M.~Wood, \emph{Universality and the circular law for sparse random matrices}, Ann. Appl. Probab. \textbf{22} (2012),
1266--1300.

\end{thebibliography}
}
\end{document}